\documentclass{amsart}
\usepackage{amsfonts, graphicx, amsmath, amssymb, color, verbatim}
\usepackage{pdfsync}
\usepackage[all, 2cell,dvips]{xypic}
\usepackage[margin = 1.5in]{geometry}

\newcommand{\p}{\partial}
\newcommand \lra {\longrightarrow}

\newcommand \absv [1]{\left \lvert #1 \right \rvert }
\newcommand \lp {\left(}
\newcommand \rp {\right)}
\newcommand \la {\langle}
\newcommand \ra {\rangle}
\newcommand{\set}[1]{\left\{ #1 \right\} }
\newcommand{\wt}[1]{\widetilde{#1}}

\newcommand\RR{\mathbb{R}}
\newcommand\NN{\mathbb{N}}
\newcommand\Id{\operatorname{Id}}

\newcommand\inc{\textrm{in}}
\newcommand\out{\textrm{out}}

\newtheorem{theorem}{Theorem}
\newtheorem{lemma}[theorem]{Lemma}
\newtheorem{proposition}[theorem]{Proposition}
\newtheorem{corollary}[theorem]{Corollary}
\theoremstyle{remark}
\newtheorem{remark}[theorem]{Remark}

\numberwithin{equation}{section}
\numberwithin{theorem}{section}

\DeclareMathOperator \Crit {Crit}

\DeclareMathOperator \sgn {sgn}
\DeclareMathOperator \dist {dist}
\DeclareMathOperator \Ai {Ai}
\DeclareMathOperator \Bi {Bi}

\title{Approximation and equidistribution of phase shifts: spherical symmetry}
\author{Kiril Datchev}
\address{Department of Mathematics, Massachusettes Institute of Technology}
\email{datchev@math.mit.edu}
\author{Jesse Gell-Redman}
\address{Mathematical Sciences Institute, Australian National University and Department of Mathematics, University of Toronto}
\email{jgell@math.toronto.edu}
\author{Andrew Hassell}
\address{Mathematical Sciences Institute, Australian National University}
\email{Andrew.Hassell@anu.edu.au}
\author{Peter Humphries}
\address{Mathematical Sciences Institute, Australian National
  University and Department of Mathematics, Princeton University}
\email{peterch@math.princeton.edu}
\thanks{The first author was partially supported by an NSF postdoctoral fellowship. The authors would like to thank Hamid Hezari, Volker Schlue and Steve Zelditch 
for very helpful conversations. The third author acknowledges the support of the Australian Research Council through a Future Fellowship FT0990895 and Discovery Grant DP1095448.}

\begin{document}

\begin{abstract}

Consider a semiclassical Hamiltonian 
\begin{equation*}
  H_{V, h} :=   h^{2} \Delta + V - E
\end{equation*}
where $h > 0$ is a semiclassical
parameter, $\Delta$ is the positive
Laplacian on $\mathbb{R}^{d}$, $V$ is a smooth, compactly
supported central potential function and $E > 0$ is an energy level. In this setting the
scattering matrix $S_h(E)$ is a unitary operator on
$L^2(\mathbb{S}^{d-1})$, hence with spectrum lying on the unit circle;
moreover, the spectrum is discrete except at $1$.

We show under certain additional assumptions on the potential that the
eigenvalues of $S_h(E)$ can be divided into two classes: a finite
number $\sim c_d (R\sqrt{E}/h)^{d-1} $, as $h \to 0$, where $B(0, R)$
is the convex hull of the support
of the potential, that equidistribute around the unit circle, and the
remainder that are all very close to $1$. Semiclassically, these are
related to the rays that meet the support of, and hence are scattered
by, the potential, and those that do not meet the support of the
potential, respectively.

A similar property is shown for the obstacle problem in the case that the obstacle is the ball of radius $R$. 
\end{abstract}

\maketitle

\section{Introduction}


\

In this paper we consider the scattering matrix for a semiclassical
potential scattering problem with spherical symmetry on $\RR^d$, $d
\geq 2$. Let $V$ be a smooth, compactly supported potential function
which is central, i.e. $V(x)$ depends only on $|x|$. We consider the
Hamiltonian 
\begin{equation} \label{eq:hamiltonian}
  H_{V, h} :=   h^{2} \Delta + V - E
\end{equation}
where $\Delta = - \sum_{i= 1}^{d} \p_{i}^{2}$ is the positive
Laplacian on $\mathbb{R}^{d}$, $E > 0$ is a positive constant (energy) and 
$h > 0$ is a semiclassical
parameter.   At the end of the introduction we will reduce to the case $E = 1$.

The scattering matrix $S_{h}(E)$ for this Hamiltonian can be defined
in terms of the asymptotics of generalized eigenfunctions of $H_{V,h}$
as follows. For each function $q_{in} \in
C^{\infty}(\mathbb{S}^{d - 1})$, there is a unique solution to $H_{V, h}u
= 0$ of the form
\begin{equation}\label{eq:geneigenfunc}
u = r^{-(d-1)/2} \lp e^{-i\sqrt{E}r/h} q_\inc(\omega) + e^{+i\sqrt{E}r/h} q_\out(-\omega) \rp + O(r^{-(d+1)/2}),
\end{equation}
as $r \to \infty$, see e.g.\ \cite{GST}.  Here $q_\out \in
C^{\infty}(\mathbb{S}^{d - 1})$. The map $q_\inc \mapsto e^{i\pi(d-1)/2} q_\out$ is
by definition the scattering matrix $S_{h}(E)$. The factor
$e^{i\pi(d-1)/2}$ is chosen so that this `stationary' definition
agrees with time-dependent definitions (see
e.g.\ \cite{RSIII} or \cite{Yafaev}), and is such that the scattering
matrix for the potential $V \equiv 0$ is the identity map. It is standard
that the scattering matrix
$S_{h}(E)$ is a unitary operator on
$L^2(\mathbb{S}^{d-1})$ for every $h > 0$, and
that, for the potentials under consideration, $S_{h}(E) - \Id$ is
compact. It follows that the spectrum lies on the
unit circle, consists only of eigenvalues, and is discrete except at
$1$. It is therefore possible to count the number of eigenvalues of
$S_{h}(E)$ in any closed interval of the unit circle not containing $1$. In
fact, semiclassically (i.e.\ as $h \to 0$) we are able to separate the spectrum of $S_{h}(E)$
into two parts. One is associated to the rays that meet the support of
the potential; to leading order in $h$ there are $c_d
(R\sqrt{E}/h)^{d-1}$ of these eigenvalues, $c_d = 2 / (d-1)!$, and the other part is associated to the
rays that do not meet the support of the potential. Those eigenvalues
corresponding to  rays that do not meet the support are close to $1$,
as one should
expect, since the eigenvalues of the zero potential are all $1$ ---
see Proposition~\ref{prop:largel} below. The
other eigenvalues are affected by the potential,  and we can ask
whether these `nontrivial' eigenvalues are asymptotically
equidistributed on the unit circle. Indeed Steve Zelditch posed this
question to one of the authors several years ago.


Before stating the main result, we discuss further the scattering
matrix in the case of central potentials. In this case the
eigenfunctions of the scattering matrix are spherical harmonics and the generalized eigenfunctions then take the form  $u = r^{-(d-2)/2}f(r)Y_{l}^{m}$, where
\begin{equation}\label{eq:1}
  \begin{split}
    \lp -\p_{r}^{2} - \frac{1}{r} \p_{r} + \frac{ \lp l + (d-2)/2
      \rp^{2}}{r^{2}} + \frac{V(r) - E}{h^{2}} \rp f &= 0.
  \end{split}
\end{equation}
Here
\begin{equation}
 f(r)= H^{(1)}_{l + (d - 2)/2}(r\sqrt{E}/h) + c(l)H^{(2)}_{l + (d - 2)/2} (r\sqrt{E}/h) \mbox{ for }
r > R,\label{eq:separated}
\end{equation}
 where the $H^{(i)}_{\nu}$ are the standard Hankel
functions, \cite{AS1964}.  With
our normalization, $S_{h}Y_{l}^{m} = c(l) Y^{m}_{l}$ with $c(l)$ from \eqref{eq:separated}
In particular, the eigenvalue of $Y_{l}^{m}$ is independent of $m$.
We write the eigenvalue corresponding to $Y_l^m$ in the form $e^{i\beta_{l,h}}$. 
The quantities $\beta_{l, h}/2$ are called `phase shifts.'   See e.g.\ \cite{RSIII} for a review of
these facts.

We now discuss conditions on the potentials in the main theorems. These conditions are dynamical conditions, i.e. conditions on the Hamiltonian dynamical system determined by the symbol of $H_{V,h}$. As usual in microlocal analysis we refer to the classical trajectories of this 
system as bicharacteristics. We first define the interaction region 
\begin{equation}
\mathcal{R} := \set{x : V(\absv{y}) < E \mbox{ for all } \absv{y} >
    \absv{x} }.
\label{intregion}\end{equation}
This is the region of $x$-space accessible by bicharacteristics coming from infinity. 
Notice that for central potentials this region takes the form 
\begin{equation}
\mathcal{R} = \{ |x| \geq r_0 \} \mbox{ where } r_0 = \inf_{r \ge 0} \{s>r \Rightarrow V(s) <E\}.
\label{r0}\end{equation}

The first condition is
\begin{equation}\begin{gathered}
\text{$V$ is nontrapping at energy $E$ in the interaction region.} 
\end{gathered} \label{nontrapping-condition}\end{equation}
That is, $x$ tends to infinity along every bicharacteristic in
$\mathcal{R}$ both forwards and backwards in time.


The second condition concerns the \textit{scattering angle}. Let $R$ be such
that $B(0, R)$ is the smallest ball containing the support of $V$,
i.e.\
\begin{equation}
B(0, R) = \mbox{chsupp} V, \mbox{ the convex hull of the
  support of V.}\label{eq:chsupp}
\end{equation}
 We recall (see
Section~\ref{sec:dynamics} for definitions and details) that for a central
potential, the scattering angle $\Sigma(\alpha)$ is a function only of
the angular momentum $\alpha$ and measures the difference between the
incident and final directions of the trajectory (which are
well-defined, since the motion is free for $|x| > R$ --- see
\cite{RSIII}. The scattering angle is zero for all trajectories with
$\alpha > R$. 
Our second condition is that 
\begin{equation}
\text{the number of zeroes of  $\Sigma'(\alpha) = \frac{d \Sigma}{d\alpha}
  (\alpha)$ in $[0, R)$ is finite.}
\label{scattering-angle-condition}\end{equation}

Then our main results are 

\begin{theorem}\label{thm:eigenvalue}
Let $R$ be as in \eqref{eq:chsupp}, and
  assume that $V \in C^{\infty}_{c}(\mathbb{R}^{d})$ is central and satisfies condition \eqref{nontrapping-condition}. 
  Define the  real-valued function $G(\alpha)$, $\alpha \in \RR$, by 
\begin{equation}
  \label{eq:galpha}
  \frac{dG}{d\alpha} (\alpha) =  \Sigma(\alpha), \quad G(\alpha) = 0 \text{ for } \alpha \ge R,
\end{equation}
where $\Sigma$ is the scattering angle function in
\eqref{Sigma}. Then  the following approximation on each eigenvalue $e^{i\beta_{l,h}}$ of $S_{h}$ is valid: 
  
  (i) If the dimension $d$ is even, then there
  exists $C = C(d)$ such that, for all $l \in \NN$ satisfying  $lh
  \leq R$, we have an estimate
  \begin{equation}
    \label{eq:eigenvalueh}
   \bigg| e^{i\beta_{l,h}} - \exp\set{ \frac{i}{h} \lp G \big((l + \frac{d-2}{2})h \big)  \rp
     }
     \bigg| \le C h    .
  \end{equation}
(ii) If the dimension $d > 2$ is odd, then for any $\epsilon > 0$ there exists $C = C(\epsilon,d)$ such that \eqref{eq:eigenvalueh} holds 
whenever $\alpha = lh \geq \epsilon$ is distance at least $\epsilon$ from the set 
\begin{equation}
\set{\alpha : \Sigma(\alpha) \in
\set{\pi k}_{k \in \mathbb{Z}}}.
\label{badalphaset}\end{equation}
\end{theorem}

\begin{theorem}\label{thm:main} 
Let $R$ be as in \eqref{eq:chsupp}, and assume that $V  \in C^{\infty}_{c}(\mathbb{R}^{d})$ is central and satisfies conditions \eqref{nontrapping-condition} and \eqref{scattering-angle-condition}.
  Then as $h \downarrow 0$, we consider the eigenvalues $e^{i
    \beta_{l,h}}$ for which $l \le R\sqrt{E}/h$, counted with multiplicity $p_d(l)  = \dim \ker \lp \Delta_{\mathbb{S}^{d - 1}} -  l(l + d - 2)
  \rp$.  There are $2 (R\sqrt{E}/h)^{d-1} /(d - 1)! +
  O(h^{- (d - 2)})$ of these, and they equidistribute around the unit circle,
meaning that
\begin{equation}
  \label{eq:maintheorem}
  \sup_{0 \le \phi_{0} < \phi_{1} \le 2\pi}\absv{\frac{N(\phi_{0}, \phi_{1})}{
   2 (R\sqrt{E}/h)^{d-1}  / (d - 1)! } - \frac{\phi_{1} -
    \phi_{0}}{2\pi} } \to 0 \mbox{  as  } h \downarrow 0,
\end{equation}
where $N(\phi_{0}, \phi_{1})$ is the number of $\beta_{l, h}$ with
  $l \le R\sqrt{E}/h$ and $\phi_{0} \le \beta_{l,h} \le \phi_{1}$ (mod
  $2\pi$), counted with multiplicity.
  \end{theorem}

\

\begin{remark} The approximation \eqref{eq:eigenvalueh} for the phase
  shifts can be found in physics textbooks; see for example
  \cite[Sect. 126]{LL1965} or \cite[Equation (18.11), Section
  18.2]{Newton1966}; it can be derived readily from the WKB
  approximation applied to \eqref{eq:1}. However, no error estimate is
  claimed in either of these sources. We have not
  been able to find any rigorous bounds on the WKB approximation of
  the phase shifts in any previous literature, so we believe the bound
  \eqref{eq:eigenvalueh} to be new. 
\end{remark}

\begin{remark}\label{rem:examples} Many potentials satisfy conditions \eqref{nontrapping-condition} and \eqref{scattering-angle-condition} --- see Section~\ref{sec:example}. 
 \end{remark}

We also show 
\begin{proposition}\label{prop:largel}
Let $V$ and $R$ be as in Theorem~\ref{thm:eigenvalue}, and let $\kappa \in (0, 1)$.
The eigenvalues $e^{i \beta_{l,h}}$ for $l \geq (R \sqrt{E} + h^\kappa)/h$ satisfy 
$$
\Big| e^{i\beta_{l,h}} - 1 \Big| = O( h^{\infty}), \quad h \to 0.
$$
Here and below, $O(h^{\infty})$ denotes a quantity that is bounded by
$C_{N}h^{N}$ for all $N$ and some $C_{N} > 0$.
\end{proposition}

\begin{remark} The methods of \cite[Section 4]{PZ2001} show that for a `black box perturbation' of the Laplacian on $\RR^d$, at most $O(h^{-(d-1)})$ eigenvalues of the scattering matrix are essentially different from $1$.
\end{remark}  

Note that the number of eigenvalues not covered by
Theorem~\ref{thm:eigenvalue} and Proposition~\ref{prop:largel} is
$o(h^{1-d})$, and hence cannot affect the equidistribution
properties. Hence we get the following equidistribution result for the
full sequence of eigenvalues of $S_{h}$.

\begin{corollary}
Suppose that $V$ satisfies  conditions \eqref{nontrapping-condition} and \eqref{scattering-angle-condition}. Then for each $\epsilon > 0$, we have 
\begin{equation}
  \label{eq:maintheoremepsilon}
  \sup_{\epsilon \le \phi_{0} < \phi_{1} \le 2\pi - \epsilon}\absv{\frac{\tilde N(\phi_{0}, \phi_{1})}{
   2 (R\sqrt{E}/h)^{d - 1} / (d - 1)! } - \frac{\phi_{1} -
    \phi_{0}}{2\pi} } \to 0 \mbox{  as  } h \downarrow 0,
\end{equation}
\textit{where $\tilde N(\phi_{0}, \phi_{1})$ is the number of $\beta_{l, h}$ (with no condition on $l$), counted with multiplicity,  satisfying $\phi_{0} < \beta_{l,h} < \phi_{1}$ (mod $2\pi$)}.
\end{corollary}

Results directly analogous to those for semiclassical potentials are also true in the case of
scattering by a disk of radius $R$ centered at the origin.  The
scattering matrix in this case can be defined similarly; given any
function $q_{\inc} \in C^{\infty}(\mathbb{S}^{d - 1})$, there is a
unique solution $u$ to the equation $\lp  \Delta - k^2 \rp u = 0$ such that\footnote{Here we prefer to use non-semiclassical notation where the energy level is $k^2$, as is traditional in obstacle scattering literature. The variable $k$ here corresponds to $1/h$ above, when the energy level $E = 1$.}
\begin{equation}\label{eq:scatsoln}
  \begin{split}
    &u = r^{- (d - 1)/2} \lp e^{-i k r} q_{\inc}(\omega) + e^{+ ik r}
    q_{\out}(-\omega) \rp + O(r^{-(d + 1)/2}), \quad r \to \infty \\
    &\qquad u \rvert_{\absv{x} = R} \equiv 0.
  \end{split}
\end{equation}
The scattering matrix $S_{k}$ is again defined $q_{\inc} \mapsto
e^{i\pi(d-1)/2} q_\out$, and the standard facts about the operator
$S_{h}$ also hold for $S_{k}$.  As above, the spherical harmonics
diagonalize the scattering matrix.  We write $S_{k} Y_{l}^{m} = e^{i
  x_{l, k}}Y_{l}^{m}$.  
  We will prove
\begin{theorem}\label{thm:disk}
  As $k \to \infty$, the eigenvalues $e^{i
    x_{l, k}}$ of $S_k$ satisfy 
\begin{equation}
\bigg| e^{ix_{l,k}} - e^{i \big(kG_b((l + (d-2)/2)/k)   +
  \pi/2\big)} \bigg| \leq Ck^{-1/2}, \quad \frac{l}{k} \leq R
 - k^{-1/3}
\label{ball-eval-approx}\end{equation}
for some uniform $C = C(d)$, where $\Sigma_b(\alpha) := -2 \cos^{-1} (\alpha/R)$  is the scattering angle for the ball and 
$G_b$ is defined by 
\begin{equation}
G_b(\alpha) = - \int_\alpha^R \Sigma_b(\alpha') \, d\alpha' = 2 \sqrt{R^2 - \alpha^2} - 2 \alpha \cos^{-1} (\alpha/R). 
\label{Gbdef}\end{equation}
The points $e^{ix_{l,k}}$ for which $l \le Rk$, counted with multiplicity $$p_d(l)
  = \dim \ker \lp \Delta_{\mathbb{S}^{d - 1}} -  l(l + d - 2)
  \rp,$$ equidistribute around $\mathbb{S}^{1}$ in the sense of 
 Theorem \ref{thm:main}.  In fact, we have the stronger statement
  \begin{equation}
    \label{eq:fastdecrease}
      \sup_{0 \le \phi_{0} < \phi_{1} \le 2\pi}\absv{\frac{N(\phi_{0}, \phi_{1})}{
   2 (Rk)^{d - 1} / (d - 1)! } - \frac{\phi_{1} -
    \phi_{0}}{2\pi} } = O(k^{-1/3})  \mbox{ as } k \to \infty.
  \end{equation}
\end{theorem}

As far as we are aware, the present paper is the first in the mathematical literature to deal with the question of the equidistribution of phase shifts over the unit circle. 
However, there are a number of previous studies of high-energy or semiclassical asymptotics of eigenvalues of the scattering matrix. The relation between the sojourn time and high-frequency asymptotics of the scattering matrix was observed in classical papers by Guillemin \cite{G1976}, Majda \cite{Maj1976} and Robert-Tamura \cite{RT1989}.  Melrose and Zworski \cite{MZ1996} showed that for fixed $h >  0$ the absolute scattering matrix for a Schr\"odinger operator on a scattering, or asymptotically conic, manifold is an FIO associated to the geodesic flow on the manifold at infinity for time $\pi$.  Alexandrova \cite{A2005} studied the scattering matrix for a nontrapping semiclassical Schr\"odinger operator, and showed that localized to finite frequency,  it is a semiclassical FIO associated to the limiting Hamilton flow relation at infinity, which includes the behavior of the Hamilton flow in compact sets. A more global description was given in Hassell-Wunsch \cite{HW2008} where the semiclassical asymptotics of the scattering matrix were unified with the singularities of the scattering matrix at fixed frequency (i.e. the Melrose-Zworski result \cite{MZ1996}).  These results are explained in Sections \ref{sec:dynamics} and \ref{sec:evals} below.

Asymptotics of phase shifts, i.e. the logarithms of the eigenvalues of the scattering matrix, were analysed by  Birman-Yafaev \cite{BY1980, BY1982, BY1984, BY1993}, Sobolev-Yafaev \cite{SY1985}, Yafaev \cite{Yafaev1990} and more recently Bulger-Pushnitski \cite{BP2011}. In \cite{SY1985}, an asymptotic form $V \sim c r^{-\alpha}$, $\alpha > 2$ was assumed and asymptotics of the individual phase shifts as well as the scattering cross section were obtained. In this paper, the strength of the potential and the energy were allowed to vary independently, so that the result includes the semiclassical limit as in the present paper. In the other papers listed above, the context was scattering theory for a fixed potential. In this setting, the scattering matrix $S(\lambda)$ tends in operator norm to the identity as $\lambda \to \infty$ so the phase shifts tend to zero uniformly. The asymptotics of individual phase shifts for a fixed energy, and also the high-energy asymptotics, were analyzed.

In the 1990s Doron and Smilansky studied  the pair correlation for phase shifts,
in particular proposing that the pair correlations should behave
statistically similarly to the (conjectural) pair correlations for
eigenfunctions of a closed quantum system: that is, the pair
correlations for chaotic systems should be the same as for certain
ensembles of random matrices, while for completely integrable systems,
they should be Poisson distributed (the `Berry-Tabor conjecture'); see
for example  \cite{Sm1992}. In \cite{ZZ1999}, Zelditch and Zworski
analyzed the pair correlation function for eigenvalues of the
scattering matrix associated to a rotationally invariant surface with
a conic singularity and a cylindrical end. They showed that a full
measure set of a  2-parameter family of such surfaces obeyed Poisson
statistics, agreeing with Smilansky's conjecture.

In a different setting Zelditch \cite{Z1997} analyzed quantized contact transformations,  which are families of unitary maps on finite dimensional spaces with dimension $N \to \infty$. He proved under the assumption that the set of periodic points of the transformation has measure zero, that the eigenvalues of these unitary operators becomes equidistributed as $N \to \infty$. 
After reading a draft of the current paper, Zelditch pointed out to the authors that a similar strategy could be used in the context of semiclassical potential scattering to prove equidistribution. In fact, this strategy is likely to be a more direct approach to proving equidistribution than the one we employ here. On the other hand, our approach has several advantages: it also gives approximations to the individual phase shifts, up to an $O(h)$ error (see Theorem~\ref{thm:eigenvalue}), and in addition appears to be a better method for obtaining a rate of equidistribution, as in Theorem~\ref{thm:disk} above. 

In future work, we plan to treat non-central potentials (or perhaps, following the suggestion of one of the referees, black box perturbations of the free Laplacian --- see \cite{SZ1991, PZ2001}) as well as non-compactly supported potentials. In the latter case, \cite{SY1985} gives some indication of what to expect; in particular, the scaling with $h$ cannot be the same as in the compactly supported case.

\noindent \textbf{Reduction to $E = 1$:}   In light of
\eqref{eq:1} and \eqref{eq:separated}, 
\begin{equation}
  \label{eq:3}
  S_{h,V}(E) = S_{\wt{h}, \wt{V}}(1),
\end{equation}
where $\wt{h} = h / \sqrt{E}$ and $\wt{V} = V / E$.  Here $S_{h,V}(E)$
denotes the scattering matrix of $h^{2}\Delta + V - E$, and
$S_{\wt{h}, \wt{V}}(1)$ denotes the scattering matrix of
$\wt{h}^{2}\Delta + \wt{V} - 1$.  For the remainder of the
paper, we assume without loss of generality that $E = 1$.

\section{Dynamics}\label{sec:dynamics}
We now review some standard material on Hamiltonian dynamics for central potentials. Consider first the case the dimension $d = 2$.

The classical Hamiltonian corresponding to our quantum system is 
$$
|\xi|^2 + V(r) - 1
$$
or in polar coordinates, using $(r, \varphi)$ and dual coordinates $(\rho, \eta)$, 
$$
H = \rho^2 + \frac{\eta^2}{r^2} + V(r) - 1,
$$
and the Hamilton equations of motion are 
\begin{equation}
\begin{aligned} 
\dot r &= 2\rho \\
\dot \rho &= -  V'(r) + 2r^{-3} \eta^2  \end{aligned}
\qquad \begin{aligned}
\dot \varphi &= 2\frac{\eta}{r^2} \\
\dot \eta &= 0.
\end{aligned}\label{eom}\end{equation}
The invariance of the Hamiltonian under rotations is reflected in the
conservation of angular momentum $\eta = 2 r^2 \dot \varphi$. For a given bicharacteristic, this is the
minimum value of $r$ along the free ($V \equiv 0$) bicharacteristic that agrees with
the given one as $t \to -\infty$ (we could just as well take $t \to
+\infty$ since it is a conserved quantity).

 \begin{figure}[htbp]
 \includegraphics[width=100mm]{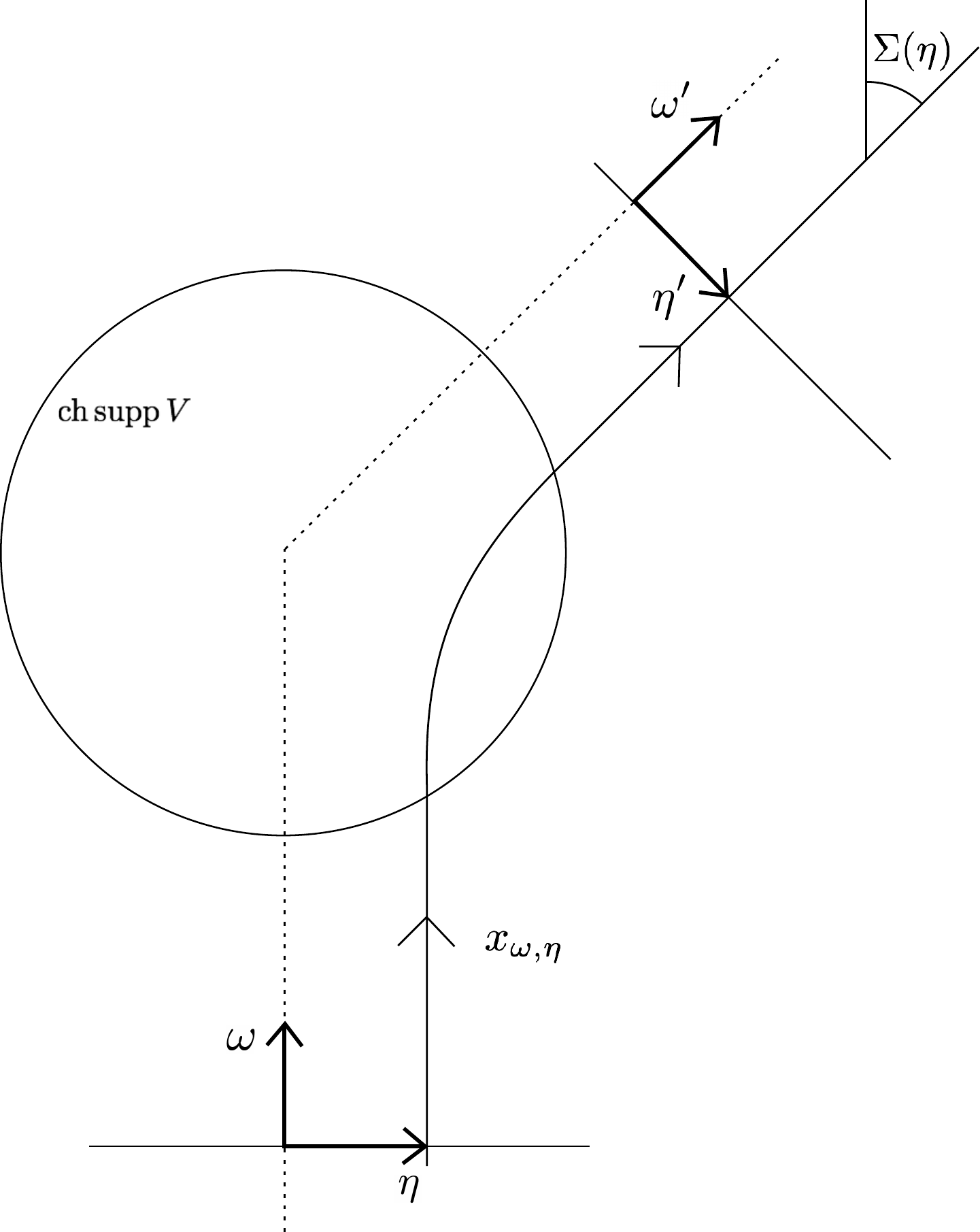}
 \caption{Here $x_{\omega, \eta}$ is the classical trajectory equal to
     $\eta + t\omega$ for $t << 0$.  The scattering angle $\Sigma(\eta)$
     is the angle between the outgoing direction $\omega'$ and the
     incoming direction $\omega$.  Note that $\absv{\eta} =
     \absv{\eta'}$ by conservation of angular momentum since the
     potential $V$ is central, $V = V(r)$.} \label{fig:scatteringangle}
 \end{figure}

Notice that in the case of general dimension $d$, each bicharacteristic lies entirely in a two-dimensional subspace, so the above discussion in fact includes the general case.

The scattering matrix is related to the asymptotic properties of the
bicharacteristic flow. Geometrically this information is contained in
a submanifold $$L \subset T^{*}\mathbb{S}^{d-1} \times
T^{*}\mathbb{S}^{d-1} \times \mathbb{R}$$ that we define
now. Returning to the case of general dimension $d$,
we identify $\mathbb{S}^{d -1}$ with the unit sphere in
$\mathbb{R}^{d}$ and identify the cotangent space
 $T^{*}_{\omega}\mathbb{S}^{d - 1}$ with
the orthogonal hyperplane $\omega^\perp$ to $\omega$.  
Given $\omega$ and $\eta \in
T^{*}_{\omega}\mathbb{S}^{d - 1}$,
take the unique bicharacteristic ray whose projection
$x_{\omega,\eta}(t)$ to
$\mathbb{R}^{d}$ 
is given by $\eta + t\omega$ for $t < < 0$. Define $(\omega', \eta')$ by  
\begin{equation}\label{eq:Cdimd}
  \begin{split}
    \omega'(\omega, \eta) &= \lim_{t \to \infty}  x(t) / \absv{x(t)} \\
    \eta'(\omega, \eta) &= \lim_{t \to \infty} x(t) - \la x(t),
    \omega' \ra \omega' 
  \end{split}
\end{equation}
and $\tau(\omega, \eta)$ to be the \textbf{sojourn time} or \textbf{time delay} for $\gamma$; this is by definition the limit
\begin{equation}
  \label{eq:sojourn}
  \lim_{a \to \infty} t_{1}(a) - t_{2}(a) - 2a = \tau(\omega, \eta),
\end{equation}
where $t_{1}(a)$ is the smallest time, $t$, for which $r(t) = a$
and $t_{2}(a)$ is the largest.  We then define $L$ to be the submanifold 
\begin{equation}
  L := \set{(\omega, \eta, \omega'(\omega, \eta), - \eta'(\omega, \eta), \tau).}
\label{Legendrian}\end{equation}
As shown in \cite{HW2008}, $L$ is a Legendrian submanifold of $T^{*}\mathbb{S}^{d-1} \times
T^{*}\mathbb{S}^{d-1} \times \RR$ with respect to the contact
form $\chi + \chi' - d\tau$, where $\chi$ is the standard contact form
on $T^* \mathbb{S}^{d-1}$, given in any local coordinates $x$ and dual
coordinates $\xi$ by $\xi \cdot dx$.  Note that the projection of $L$
to $T^{*}\mathbb{S}^{d-1} \times
T^{*}\mathbb{S}^{d-1}$ is Lagrangian with respect to the standard symplectic form. 
Indeed it is the graph of a symplectic transformation $(\omega, \eta)
\mapsto (\omega', \eta')$, and the scattering matrix is a semiclassical Fourier integral operator associated to this symplectic graph \cite{A2005}, \cite{HW2008}. The sojourn time,
however, carries extra information and is directly related to
high-energy scattering asymptotics as observed in \cite{Maj1976},
\cite{G1976}, \cite{HW2008}.


The previous paragraph applies to any potential, central or not. We
return to the case of a central potential $V$, for which, as observed
above, the dynamics take place in a two-dimensional subspace, so we
can assume $d=2$ without loss of generality. In that case we use the
angular variables $\varphi, \varphi'$ in dimension $d=2$ instead of
$\omega, \omega'$ above. Consider a bicharacteristic $\gamma$ with angular
momentum $\eta \in \RR$, initial direction $\varphi \in
\mathbb{S}^1$ and final direction $\varphi' \in \mathbb{S}^1$. 
The scattering angle $\Sigma(\eta)$ determined by $V$ is, by definition,  the angle
between the initial and final directions of $\gamma$,  normalized
so that $\Sigma$ is continuous and $\Sigma(\eta) = 0$ for $\eta > R$, i.e. 
\begin{equation}
\Sigma(\eta) = \varphi'(\varphi, \eta) - \varphi, \quad \Sigma(\eta) = 0 \text{ for } \eta > R. 
\label{Sigma}\end{equation}
See Figure \ref{fig:scatteringangle}.  Note that $\Sigma$ is
independent of $\varphi$ by rotational invariance  of $V$. 
In the central case, there is a standard expression for $\Sigma(\eta)$ in terms of the potential (see for example \cite[Section 5.1]{Newton1966}), that we now derive. 
 Indeed, if $V$ is central and non-trapping at energy $1$,
then along a bicharacteristic, the functions $\rho$ and $\dot \rho$ do
not have simultaneous zeros.  For if there were such a time, and
the value of $r$ at this time were $r_{0}$, then $r(t) \equiv r_{0}$
would be a bicharacteristic, contradicting the non-trapping
assumption.  Hence the zeros of the function $1 -
\eta^{2} / r^{2} - V(r)$ are simple on the region of
interaction \eqref{intregion}. 
  Given a fixed bicharacteristic, let $r_{m}$ be the minimum value of
  $r$; note that $r_m$ is a function only of the angular momentum $\eta$. We denote the derivative of $r_m$ with respect to $\eta$ by $r_m'(\eta)$.   By symmetry, $r_{m}$ is the unique
value of $r$ along the trajectory at which $\rho = 0$ and $\dot \rho > 0$, so we can divide the
bicharacteristic `in half,' and consider only times when $r > r_{m}$
and $\rho \ge 0$.  For such times $r$ is a strictly monotone function of $t$, and  we have 
\begin{align*}
   \frac{d\varphi}{dr} &= \frac{d\varphi}{dt} \frac{dt}{dr} =  \frac{\eta}{r^{2}} \frac1{\rho} = \frac{\eta}{r^{2} \sqrt{1 - \eta^{2}/r^{2} - V(r)}}.
\end{align*}
By the simplicity of the zeros in the denominator, we can integrate to
obtain, for $\eta > 0$, 
\begin{align}\label{eq:integral}
    \Sigma(\eta) &= \pi -  2 \int_{r_{m}}^{\infty} \frac{\eta}{r^{2} \sqrt{1
      - \eta^{2}/r^{2} - V(r)}}dr.
\end{align} 

  The sojourn time is also independent of
$\varphi$, and we write
\begin{equation}
T(\eta) = \tau(\varphi, \eta).
\label{T}\end{equation}
Notice that both $\Sigma$ and $T$ depend only on $\eta$ in the central case.  The fact that $L$ in \eqref{Legendrian} is Legendrian
then implies the following relation between these functions:
\begin{equation}
d\tau = \eta \cdot (d\varphi - d\varphi') \implies \frac{d}{d\eta}
T(\eta) = -\eta  \frac{d}{d\eta}{\Sigma}(\eta). 
\label{tau-eta}\end{equation}

\begin{remark} Notice that the ambiguity of $\Sigma$ modulo $2\pi$ is eliminated by our convention that $\Sigma(\eta) = 0$ for $\eta > R$. We point out that by reflection symmetry, we have $\Sigma(\eta) = - \Sigma(-\eta)$ modulo $2\pi$, but it might not be the case that $\Sigma(\eta) = - \Sigma(-\eta)$ on the nose: this will happen if and only if $\Sigma(0) = 0$, which will  be the case if and only if the interaction region is the whole of $\RR^d$.  However, we always have $\Sigma'(\eta) = \Sigma'(-\eta)$, which shows that $T'(\eta)$ is an odd function, and hence $T(\eta)$ is even in $\eta$. 
\end{remark}


\section{Asymptotic for the eigenvalues of $S_{h}$}\label{sec:evals}
In this section we prove Theorem~\ref{thm:eigenvalue}, that is, the
error bound \eqref{eq:eigenvalueh} for the asymptotics of the
eigenvalues $e^{i\beta_{l,h}}$ of $S_{h}$.  To do this, we will use a
Fourier integral approach. One could also directly attack \eqref{eq:1}
using ODE methods; see Remark~\ref{FIOvsODE} for further discussion on
this point. 

We use the fact, proven in
\cite{HW2008}, \cite{A2005},
that the integral kernel of $S_{h}$ is an oscillatory integral
associated (in a manner we describe directly) to the Legendre submanifold $L$ in \eqref{Legendrian}. To be precise, the Schwartz kernel of $S_{h}$ can be decomposed
following \cite[Prop. 15]{HW2008}
(with minor changes in notation) as
\begin{equation*}
  S_{h} = K_{1} + K_{2} + K_{3},
\end{equation*}
with the $K_{i}$ as follows.  

Fix $R_2 > R_1 > R$.
First, $K_2$ is a pseudodifferential operator of order zero (both in the sense of semiclassical order and differential order), microsupported in $\{ |\eta| > R_1\}$, hence taking the form in local coordinates $z$ on $\mathbb{S}^{d-1}$
$$
(2\pi h)^{-(d-1)} \int e^{i(z - z') \cdot \zeta/h} b(z, \zeta, h) \, d\zeta
$$
for some smooth symbol $b(z, \zeta,h)$ equal to zero for
$|\zeta|_{g(z)} < R_1$ where $| \cdot |_{g(z)}$ is the standard norm on
$T_z^*\mathbb{S}^{d-1}$.  This reflects
the fact that the Legendrian submanifold $L$ in \eqref{Legendrian} is
the diagonal relation $\omega = \omega', \eta = -\eta', \tau = 0$ for
$|\eta|, |\eta'| > R$, to which pseudodifferential operators are
associated. 
Moreover, $K_2$ is microlocally equal to the identity for $|\eta| >
R_2$, i.e. $b = 1 + O(h^\infty)$ for $|\zeta|_g > R_2$. Indeed, the full symbol (up to $O(h^\infty)$) of the
scattering matrix is determined by transport equations along the rays
with $|\eta| > R$. 
Since these transport equations are identical to those for the zero potential, the scattering matrix in this microlocal region is microlocally identical to that for the zero potential, which is the identity operator.

Next, $K_1$ is a semiclassical Fourier integral operator of semiclassical order 0 with compact microsupport in $\set{ \absv{\eta} < R_2 }$. That is, $K_1$ is given by a sum of terms taking the form in local coordinates 
\begin{equation}\label{eq:u1}
  K_{1}(\omega, \omega', h)  = h^{-(d - 1)/2 - N/2} \int_{\mathbb{R}^{N}}
  e^{i \Phi(\omega, \omega', v)/h} a(\omega, \omega', v, h) dv
  \end{equation}
  with respect to a suitable phase function $\Phi$ and smooth compactly supported function $a$. Here the phase function parametrizes $L$ locally, meaning
\begin{enumerate}
\item On the set $\Crit \Phi := \set{(\omega, \omega', v):
    D_{v}\Phi(\omega, \omega', v) = 0}$, $D_{\omega,
    \omega', v} \Phi$ has rank $N$.  This implies that
    \begin{equation}\label{eq:phaseparam}
    	L(\Phi) := (\omega, D_{\omega}\Phi(\omega, \omega', v),
  \omega' , D_{\omega'}\Phi(\omega, \omega', v), \Phi(\omega, \omega',
  v))
	\end{equation}
is a smooth submanifold.
\item 
$L(\Phi) =   L$ at points for which $ a \neq O(h^{\infty})$.
\end{enumerate}
By $K_{1}$ having compact microsupport in the set $\set{ \absv{\eta} < R_2 }$, we mean specifically that if
$(\omega, \eta, \omega'
, \eta', \tau)$ 
$\in L$ has $\absv{\eta} = \absv{\eta}' > R_2$ and \ $(\omega, \omega',
v) \in \Crit (\Phi)$ with $(D_{\omega,
  \omega'}\Phi(\omega, \omega', v_{i}), \Phi )= (\eta, \eta', \tau)$, then
$a(\omega, \omega', v,h) = O(h^\infty)$ in a neighbourhood of $(\omega, \omega',
v)$.


Finally, $K_3$  is a kernel in $\dot{C}^{\infty}(\mathbb{S}^{d -1} \times \mathbb{S}^{d
  -1} \times [0, h_{0}))$, i.e.\ smooth
and vanishing to all orders at $h = 0$.  

For the proof of Theorem \ref{thm:eigenvalue} we need to know the
principal symbol of $K_1$ as a semiclassical FIO.  By
\eqref{eq:phaseparam}, the canonical relation of $K_{1}$, $C$, is the
projection of $L$ off the $\mathbb{R}$ factor, i.e.\ onto
$T^{*}\mathbb{S}^{d-1} \times T^{*}\mathbb{S}^{d-1}$.  Precisely, with 
notation as in \eqref{Legendrian}, 
\begin{equation}
  \label{eq:canonicalrelation}
  C = \set{ (\omega, \eta, \omega', -\eta')}. 
\end{equation}

\begin{lemma}\label{lem:prsymb1} The Maslov bundle of the canonical
  relation $C$ of the FIO $K_1$ is canonically trivial, and with
  respect to this canonical trivialization, the principal symbol of
  $K_1$ is equal to $1$, as a multiple of the Liouville half-density
  on $C$ coming from either the left or right projection of $C$ to
  $T^* \mathbb{S}^{d-1}$.  That is to say, 
  \begin{equation}
    \label{eq:principalsymbolgeneral}
    \sigma(K_{1})(\omega, \eta, \omega', -\eta') = \absv{ d \omega\,
      d\eta}^{1/2} = \absv{d\omega'\, d\eta'}^{1/2} ,
  \end{equation}
for $(\omega, \eta, \omega', -\eta') \in C$ such that $\absv{\eta} \le
R_{1}$. (The two half-densities in \eqref{eq:principalsymbolgeneral}
are equal on $C$ since $C$ is a Lagrangian submanifold.)
\end{lemma}

\begin{proof} Consider first the Maslov bundle of $C$. Notice that $C$ is almost the same as $L$; in fact, it is given by 
$$
C = \{ (\omega, \eta, \omega', - \eta') \mid  \exists \, \tau \text{ such that } (\omega, \eta, \omega', \eta', \tau) \in L \}.
$$
Since $C$ is a canonical graph (i.e.\ the graph of a
symplectomorphism), associated to the scattering relation as in
\eqref{eq:Cdimd}, it projects diffeomorphically to $T^*
\mathbb{S}^{d-1}$ via both the left and right projections, and the
lift of Liouville measure on $T^* \mathbb{S}^{d-1}$ via the left
projection agrees with the lift via the right projection (since $C$ is a Lagrangian submanifold
of $T^*
\mathbb{S}^{d-1} \times T^*
\mathbb{S}^{d-1}$ and the Liouville measure can be expressed in terms of the symplectic
form $T^*
\mathbb{S}^{d-1}$), providing a
canonical half-density on $C$. We also note that the scattering relation is the identity whenever $|\eta| \geq R$ since then the corresponding bicharacteristic is not affected by the potential. Therefore, over this part of $C$ there is a canonical trivialization of the Maslov bundle. Since the Maslov bundle is flat, we can use parallel transport to extend this to a global trivialization: in fact, in the case $d=2$, the space $T^* \mathbb{S}^1$ retracts to $\mathbb{S}^{d-1} \times \{ \eta > R \}$, while for $d \geq 3$, $T^* \mathbb{S}^{d-1}$ is simply connected, hence in either case parallel transport provides an unambiguous trivialization. 

We now consider the principal symbol of the scattering matrix. 
The scattering matrix may be viewed as a `boundary value' (after
removing a vanishing factor and an oscillatory term) of the Poisson
operator, as in \cite[Section 7.7 and Section 15]{HW2008}. The
principal symbol of the scattering matrix is correspondingly derived from the principal symbol of the Poisson operator. The principal symbol of the Poisson operator is real: it solves a real transport equation with initial condition $1$. Therefore, the principal symbol of the scattering matrix is real, up to Maslov factors, i.e. it is a real number times an eighth root of unity. On the other hand, unitarity of the scattering matrix shows that the principal symbol lies on the unit circle (as a multiple of the canonical half-density); hence it is an eighth root of unity. 
Finally, the principal symbol of the scattering matrix is equal to $1$ for $|\eta| \geq R$, since here the scattering matrix is microlocally equal to the scattering matrix for the zero potential, which is certainly equal to $1$. Since the principal symbol is smooth, is restricted to eighth roots of unity, and is $1$ for $|\eta| \geq R$, it follows that the principal symbol is equal to $1$ everywhere. 
\end{proof}

\begin{proof}[Proof of Theorem \ref{thm:eigenvalue}:]

First we reduce the problem to the cases $d = 2$ and $d = 3$ as follows.
Writing  $\beta_{l, h, d}$ for the eigenvalue $\beta_{l,h}$ in dimension $d$, observe that by \eqref{eq:1},
 \begin{equation}\label{eq:dimensiondiff}
\beta_{l, h, d + 2k} = \beta_{l + k, h, d} \mbox{ for } d \geq 2, k \geq 0. 
\end{equation}
It follows that for $d \geq 4$ even, we have $\beta_{l, h, d} = \beta_{l + (d-2)/2, h, 2}$ and for $d \geq 5$ odd, we have $\beta_{l, h, d} = \beta_{l + (d-3)/2, h, 3}$.

Consider the case dimension $d=2$. For any smooth function $G \colon \mathbb{R} \lra \mathbb{R}$, the function
\begin{equation}
\Phi(\varphi, \varphi', v) = (\varphi - \varphi') v + G(v),\label{eq:phi2d}
\end{equation}
parametrizes the Legendrian (see \eqref{eq:phaseparam})
\begin{equation}\label{eq:localparamer}
L(\Phi) := \big\{ (\varphi, \eta, \varphi', \eta', \tau) \mid \eta = v = -\eta', \varphi' - \varphi = \frac{dG}{dv}(v), \ \tau = -v \frac{dG}{dv}(v)  + G(v) \big\}.
\end{equation}

With $G$ as in \eqref{eq:galpha}, this gives an explicit global parametrization of the Legendrian submanifold $L$ in
\eqref{Legendrian} if we take $\varphi \in [0, 2\pi], \varphi' \in
\mathbb{R}$.  In this case the relation between $\tau$ and $\Sigma$
given by the last equation in \eqref{eq:localparamer} is
$$
\tau =  -\eta \Sigma(\eta) + G(\eta) \implies \frac{d}{d\eta} \tau =  -\eta \frac{d}{d\eta} \Sigma(\eta),
$$
in agreement with \eqref{tau-eta}. Therefore, plugging \eqref{eq:phi2d} into \eqref{eq:u1}, the operator $K_1$ takes the form 
$$
K_{1}(\varphi, \varphi', h) = (2\pi h)^{-1} \int_{\RR} e^{i\lp(\varphi
  - \varphi')v + G(v)\rp/h} a(\varphi - \varphi',  v, h) \, dv, \quad
\varphi \in [0, 2\pi], \varphi' \in \mathbb{R},
$$
where $a$ is smooth and supported in $|v| \leq R_2$. 
Notice that we may assume that $a$ depends only on $(\varphi - \varphi', v,h)$ since the scattering matrix and the phase function both have this property.

Now we obtain an expression for the eigenvalue $e^{i\beta_{l, h}}$ of the scattering matrix $S_h$ on $Y_l = (2\pi)^{-1/2} e^{il\varphi}$ using
\begin{equation}
e^{i\beta_l, h} = \la S_{h}Y_{l}, Y_{l} \ra = \la K_{1}Y_{l}, Y_{l} \ra + \la K_{2}Y_{l}, Y_{l} \ra + \la K_{3}Y_{l}, Y_{l} \ra.  
\label{betaK}\end{equation}
  Clearly $\la K_{3}Y_{l}, Y_{l} \ra = O(h^\infty)$. Consider the $K_1$ term. 
  Writing $l = \alpha/h$ gives
  $$
  \la K_{1}Y_{l}, Y_{l} \ra = (2\pi h)^{-1}  (2\pi)^{-1}
  \int_{\mathbb{R}} \int_{0}^{2\pi} \int_{\mathbb{R}} e^{i\lp(\varphi
    - \varphi')v + G(v) - \alpha(\varphi - \varphi')\rp/h} a(\varphi -
  \varphi', v, h) \, dv\, d\varphi \, d\varphi'.  
 $$
Changing integration variables to $(\varphi,  \tilde\varphi = \varphi - \varphi')$, the
kernel is independent of the first of these variables, so that
integrating in it simply removes the factor $2\pi$. We are left with 
 $$
  \la K_{1}Y_{l}, Y_{l} \ra = (2\pi h)^{-1}  \int_{\mathbb{R}} \int_{\mathbb{R}} e^{i\lp\tilde\varphi v + G(v) - \alpha \tilde\varphi \rp/h} a(\tilde\varphi, v, h) \, dv \, d\tilde\varphi .  
 $$
 The phase is stationary at the point $v = \alpha$, $-\tilde\varphi = G'(v) = \Sigma(v)$ and the stationary phase lemma shows that the integral is equal to 
\begin{equation}
  \la K_{1}Y_{l}, Y_{l} \ra = e^{i G(\alpha)/h} a(-\Sigma(\alpha), \alpha, 0)+ O(h)
\label{eq:sp1}\end{equation}
 (noting that the Hessian of the phase function has determinant $1$
 and signature $0$).

Next we write
$$
\la K_{2}Y_{l}, Y_{l} \ra = (2\pi h)^{-1} (2\pi)^{-1} \int e^{i(\varphi - \varphi') v/h} b(\varphi,
v, h) e^{-i\alpha (\varphi - \varphi')/h}  \, dv \,
d\varphi \, d\varphi'.
$$
Here, the phase is stationary when $\alpha =v$. However, $b$ is
supported where $|v| \geq R_1 > R$ while $\alpha \leq R$ by
hypothesis, so there are no stationary points on the support of the
integrand. It follows that $\la K_{2}Y_{l}, Y_{l} \ra = O(h^\infty)$.
Thus by \eqref{betaK}
\begin{equation}
e^{i\beta_{l, h}} = e^{i G(\alpha)/h}a(-\Sigma(\alpha), \alpha, 0)
+ O(h).
\label{prebetalh}\end{equation}

The principal symbol of $K_{1}$ as an FIO is given as a multiple of the Liouville half-density on $T^* \mathbb{S}^1$, $ |d\varphi \, d\eta|^{1/2}$, by  \cite[Section 3]{H1971}
\begin{equation}
\sigma(K_1)(\varphi, \eta, \varphi + \Sigma(\eta), -\eta) = a(-\Sigma(\eta), \eta, 0) |d\varphi \, d\eta|^{1/2}.
\label{eq:princsymbd=2}\end{equation}
Indeed, the density $d_{C}$ defined on page 143 of that paper equals
$\absv{d\varphi \, d\eta}$, where we used coordinates $(x, \theta) = (\varphi,
\varphi', v)$.  The principal symbol is the image of the map from $C$ to
$\Lambda$ defined immediately following the definition of $d_{C}$, in
the notation of that paper.  In the notation of the current paper, $C
= \mbox{Crit}(\Phi)$ and $\Lambda$ is the projection of the Legendrian
$L$ onto the first four coordinates, i.e.\ it is $C$ from \eqref{eq:canonicalrelation}.  It
follows from equation~\eqref{eq:princsymbd=2} and equation \eqref{eq:principalsymbolgeneral} that
\begin{equation}
 a(-\Sigma(\eta), \eta, 0)  = 1.
 \label{eq:F2}\end{equation}
Combining \eqref{eq:F2} and \eqref{eq:sp1}, we see that 
\begin{equation}
e^{i\beta_{l, h}} = e^{i G(\alpha)/h} 
+ O(h), 
\label{betalh}\end{equation}
establishing  \eqref{eq:eigenvalueh}.


 We proceed to the case $d=3$. 
 In this case, we will obtain the
 eigenvalue $e^{i\beta_{l,h}}$ by pairing the scattering matrix $S_h$
 with the highest weight spherical harmonics
  $Y_{l}^{l}$.  These concentrate along a great circle $\gamma$, which
  we parametrize by arclength, $\varphi \in [0, 2\pi]$.  Choose Euclidean
  coordinates in $\RR^3$ so that the two-plane spanned by $\gamma$ is
  the plane $x_3 = 0$. Then $Y^l_l= c_{l}  (x_1 + i
  x_2)^l $ where $c_{l}$ is a normalization factor, equal to
  $(2\pi)^{-1/2}(\pi l)^{1/4}(1 + O(l^{-1}))$.   Let $\theta$ be the
  spherical coordinate equal to the angle with the positive $x_{3}$
  axis. Then we can write 
  \begin{equation}
    \label{eq:gaussianbeam}
    Y_{l}^{l}(\varphi, \theta) = c_{l} e^{il\varphi} (\sin \theta)^l = c_{l} e^{il\varphi} e^{-l g(\theta)}
  \end{equation}
  where $g(\theta) = - \log \sin \theta = (\theta - \pi/2)^2/2 + O((\theta - \pi/2)^4)$.

  In particular,  expression \eqref{eq:gaussianbeam} shows (and it is
  in any case well known) that the $Y^l_l$ concentrate semiclassically
  at the set $\{ \theta = \pi/2, \zeta = 0,
  \sigma = \alpha \}$ where $l = \alpha/h + O(1)$. Here we use
  coordinates $(\sigma, \zeta)$ dual to $(\varphi, \theta)$.  To
  compute the pairing \eqref{betaK} with $Y_{l}^{l}$ replacing $Y_{l}$, we first need
  to determine an oscillatory integral expression for $K_1$ that is
  valid in this microlocal region. (Note that the $K_2$ and $K_3$
  terms give an $O(h^\infty)$ contribution as before.) So choose
  $\alpha_0$ distance $\geq \epsilon$ from the set
  \eqref{badalphaset}.  As we will see, it suffices to find a local parametrization of $L$ in a neighbourhood of 
$$
\{ \theta = \theta' = \pi / 2, \ \zeta = \zeta' = 0,  \ \sigma = -\sigma' = \alpha_0, \ \varphi - \varphi' = \Sigma(\alpha_0) \};
$$
this is the set of incoming and outgoing data of bicharacteristics
with angular momentum $\alpha_0$ (see Section \ref{sec:dynamics})  which remain in the $x_{3} = 0$
plane.  To define this parametrization, we consider first a
parametrization in two dimensions locally near a bicharacteristic
with angular momentum $\eta = \alpha$. As we have seen such a two dimensional parametrization is $(\varphi - \varphi') v + G(v)$, for $v$ close to $\alpha$. 
We note that when $v = \alpha$,  $\varphi' - \varphi=\Sigma(\alpha)$, and we can write it in the form
\begin{equation}
\varphi' - \varphi = \pm \dist(\varphi, \varphi') + 2\pi k
\label{eq:pm}\end{equation}
for some integer $k$ (recalling that the distance $\dist(\varphi, \varphi')$ lies strictly between $0$ and $\pi$). We now claim that  a suitable phase function is 
 \begin{equation}
  \label{eq:phase3d}
  \Phi(\omega, \omega', v) =  (\mp\dist(\omega, \omega') - 2\pi k) v + G(v),
\end{equation}
where $v \in \mathbb{R}$ is localized near $\alpha_0$, $G(v)$ is as in \eqref{eq:galpha},  and the sign $\mp$ and the value of $k$ agree with the two-dimensional case. Indeed, on each two-plane, if we use spherical coordinates $(\overline\varphi, \overline\theta)$ adapted to that 2-plane then the form of the phase function agrees by construction with the two-dimensional phase function and therefore parametrizes that part of $L$ associated to that 2-plane (since the dynamics on each 2-plane is identical to the $d=2$ dynamics), that is, the subset (in the coordinates adapted to that 2-plane, indicated by a bar)
\begin{equation}
\{ \overline\theta = \overline\theta' = \pi/2, \overline\zeta = \overline\zeta' = 0, \overline\varphi' - \overline\varphi = \Sigma(\alpha), \overline\sigma = -\overline\sigma' = \alpha, \tau = T(\alpha) \}. 
\label{Lplane}\end{equation}

We now observe that we can eliminate $k$ by redefining $G(v)$ locally to be $G(v) + 2\pi k v$, which only has the irrelevant effect of changing $\Sigma$ by $2\pi k$ (notice also that this does not affect the eigenvalue formula  involving $e^{iG(lh)/h}$ in the statement of Theorem~\ref{thm:eigenvalue}). 
From here on we only work with the $+$ sign in \eqref{eq:pm}, i.e. the $-$ sign in \eqref{eq:phase3d}, and $k=0$. Notice that this  means that $0 < \varphi' - \varphi < \pi$ and $0 < \Sigma(\alpha) < \pi$, i.e. $\sin \Sigma(\alpha) > 0$. 
\
Returning to our spherical coordinates associated to the 2-plane $x_3 = 0$, we can use  the spherical cosine law applied to the spherical triangle with vertices  $(\varphi, \theta)$, $(\varphi', \theta')$, and the pole $x_3 = 1$:
$$
\cos \dist((\varphi, \theta),(\varphi', \theta')) =
\cos(\varphi-\varphi')  \sin \theta \sin \theta' + \cos \theta \cos
\theta'
$$
to write 
\begin{equation}
\Phi(\varphi, y, \varphi', y, v) = -\cos^{-1} \lp \cos(\varphi-\varphi')  \sin \theta \sin \theta' + \cos \theta \cos \theta' \rp v + G(v).
\label{cosinelaw}\end{equation}
We can then write in these coordinates
\begin{equation}\begin{aligned}
L = \Big\{ &(\varphi, \theta, \varphi', \theta', \sigma, \zeta, \sigma', \zeta', \tau) \mid \\
\sigma &= \p_{\varphi}\Phi = -\frac{v}{\sin \dist(\omega, \omega')} \lp
\sin(\varphi-\varphi') \sin \theta \sin \theta' \rp  \\
\zeta &= \p_{\theta}\Phi = \frac{v}{\sin \dist(\omega, \omega')}  \lp \cos(\varphi-\varphi') \cos \theta \sin \theta' - \sin \theta \cos \theta' \rp  \\
\sigma' &=  \p_{\varphi'}\Phi =\frac{v}{\sin \dist(\omega, \omega')} \lp
\sin(\varphi-\varphi') \sin \theta \sin \theta' \rp  \\
\zeta' &= \p_{\theta'}\Phi = \frac{v}{\sin \dist(\omega, \omega')}  \lp \cos(\varphi-\varphi') \cos \theta' \sin \theta - \sin \theta' \cos \theta \rp  \\
\tau &= \dist(\omega, \omega') v + G(v) \Big\} \quad 
\text{where } \dist(\omega, \omega')  =  G'(v). 
\end{aligned}\label{Lcoords}\end{equation}
(Notice that by direct inspection we see that this agrees with \eqref{Lplane} when $\theta = \theta' = \pi/2$, since then $\cos \theta = \cos \theta' = 0$ and $\dist(\omega, \omega') = \varphi' - \varphi = \Sigma(\alpha)$ and so $\sigma = v = \alpha$.) 

The scattering matrix, microlocalized to this region of phase space,  will then take the form 
\begin{equation}
(2\pi h)^{-3/2} \int e^{i\Phi(\omega, \omega',v)/h} a(\omega, \omega', v, h) \, dv. 
\label{vparam}\end{equation}

 
In terms of this parametrization   the principal symbol of \eqref{vparam}, say where both $\omega$ and $\omega'$ lie near the great circle $\gamma$ and hence where we can use coordinates $(\varphi, \theta, \varphi', \theta'; \sigma, \zeta, \sigma', \zeta', \tau)$,  is given at the point $(\varphi, \pi/2, \varphi + \Sigma(\alpha)h, \pi/2, \alpha, 0, -\alpha, 0, \tau(\alpha))$ by  \cite{H1971}
\begin{equation}
a(\varphi, \pi/2, \varphi + \Sigma(\alpha), \pi/2, \alpha, 0) e^{-i\pi/4} \big| ds \, d\theta\,  d\sigma\,  d\zeta \big|^{1/2} \Big|\det \frac{ \partial (\varphi, \theta, \sigma, \zeta, d_v \Phi)}{\partial (\varphi, \theta, \varphi', \theta', v)} \Big|^{-1/2}
\label{eq:princsymbd=3}\end{equation}
where the $e^{-i\pi/4}$ is a Maslov factor; see Remark~\ref{Maslovd=3} for more discussion about this. 
We need to compute the determinant above. We can disregard the repeated coordinates $(\varphi, \theta)$ and compute, using \eqref{Lcoords},  
\begin{equation}
\det \frac{ \partial (\sigma, \zeta, d_v \Phi)}{\partial (\varphi', \theta', v)} 
= \det \begin{pmatrix} 0 & 0 & -1 \\ 0 & \frac{-v}{\sin (\varphi'-\varphi)} & 0 \\
-1 & 0 & G''(v) \end{pmatrix} = \frac{v}{\sin(\varphi'-\varphi)} =
\frac{\alpha}{\sin \Sigma(\alpha)} \text{ at } \theta = \theta' = \frac{\pi}{2}.
\end{equation}
It follows that the principal symbol is 
\begin{equation}
a(\varphi, \pi/2, \varphi + \Sigma(\alpha), \pi/2, \alpha, 0) e^{-i\pi/4} \Big( \frac{\alpha}{\sin \Sigma(\alpha)} \Big)^{-1/2}  \big| ds \, d\theta\,  d\sigma\,  d\zeta \big|^{1/2}. 
\end{equation}
Then by equation \eqref{eq:principalsymbolgeneral}
\begin{equation}
a(\varphi, \pi/2, \varphi + \Sigma(\alpha), \pi/2, \alpha, 0) e^{-i\pi/4} = 
\Big( \frac{\alpha}{\sin \Sigma(\alpha)} \Big)^{1/2}  . 
\label{unitarity}\end{equation}

We next write the 
 contribution of $K_1$ to the expression \eqref{betaK} for the
 eigenvalue $e^{i\beta_{l,h}}$. Writing $l = \alpha/h$ and using
 \eqref{eq:gaussianbeam} we get  
\begin{equation}\label{eq:badphase}
  \begin{split}
  \la K_{1}Y^{l}_{l}, Y^{l}_{l} \ra &=  (2\pi h)^{-3/2} \int e^{i\Phi(\varphi, \theta, \varphi', \theta', v)/h}
    e^{-i\alpha(\varphi-\varphi')/h} \lp \frac{\alpha}{\pi h} \rp^{1/2}
    (2\pi)^{-1} \\ & \quad \qquad \qquad \times e^{-\alpha g(\theta)/h} e^{-\alpha g(\theta')/h} a(\varphi,
    \theta, \varphi', \theta', v, h) \, ds \, d\varphi' \, d\theta \,
    d\theta' \, 
    dv \lp 1 + O(h)\rp
  \end{split}
\end{equation}
Here the factors $(\alpha/\pi h)^{1/2} (2\pi)^{-1}$ are to normalize the functions $Y^l_l$ in $L^2$. 
We will analyze this using the stationary phase lemma with complex phase function, see
e.g. \cite[thm.\ 7.7.5]{Hvol1}.  Here the phase is 
\begin{equation}\label{eq:goodphase}
\Psi(\varphi, \theta, \varphi', \theta', v) = \Phi - \alpha(\varphi-\varphi') + i\alpha(g(\theta) + g(\theta')). 
\end{equation}
Notice that the integrand as a function of $(\varphi,\varphi')$ depends only on $\varphi-\varphi'$ by the rotational invariance of the scattering matrix, and the form of the $Y^l_l$ which take the form $e^{il\varphi}$ times a function of $\theta$. We change variable to $(\varphi, \tilde \varphi)$, $\tilde \varphi = \varphi -\varphi'$ and integrate out the variable $\varphi$, giving us a factor of $2\pi$. Then $\Psi$ has nondegenerate stationary points in the remaining variables $(\tilde \varphi, \theta, \theta', v)$. 
The imaginary part of the phase is
stationary only at $\theta = \theta' = \pi/2$, while stationarity of the real part
requires that $v = \alpha$ and $-\tilde \varphi = G'(v) = \Sigma(\alpha)$. 
The stationary phase lemma then gives us that \eqref{eq:badphase} is equal to 
\begin{equation}
  \label{eq:statphase}
  \begin{split}
& 2\pi  \lp (2\pi h)^{-3/2} \lp \frac{\alpha}{\pi h}\rp^{1/2}  (2\pi)^{-1}  \rp (2\pi h)^2 
\\ & \qquad \times \lp e^{iG(\alpha)/h} \frac1{\det (-iD^2
  \Psi)^{1/2}} a(\varphi, \pi/2, \varphi + \Sigma(\alpha), \pi/2, \alpha, 0) +
O(h) \rp.
\end{split}
 \end{equation}
 Here, to keep track of constants, we have written out all constants in \eqref{eq:badphase}; the first $2\pi$ comes from the integral in $\varphi$ and the $(2\pi h)^2$ comes from the leading term in stationary phase in the four variables $(\tilde \varphi, \theta, \theta', v)$. Simplifying the constants and using \eqref{unitarity} this is equal to 
\begin{equation}
  \label{eq:statphase2}
\lp\frac{2\alpha^2}{\sin \Sigma(\alpha)} \rp^{1/2}  
\Bigg(  \frac{e^{iG(\alpha)/h}   }{(\det -iD^2 \Psi)^{1/2}}  + O(h) \Bigg).
 \end{equation}

We will show that, in the above expression
\begin{equation}\label{eq:hessian}
\det  -i D^{2}\Psi (\varphi, 0, \varphi + \Sigma(\alpha), 0, \alpha) = \frac{2 i \alpha^{2}}{\sin
  \Sigma(\alpha)} e^{-i \Sigma(\alpha)}.
\end{equation}
Accepting this for the moment, we obtain from \eqref{eq:statphase}
$$
e^{i\beta_{l,h}} = e^{iG(\alpha)/h} e^{i\Sigma(\alpha)/2}  + O(h) . 
$$
Since $\Sigma(\alpha) = G'(\alpha)$ and $\alpha = lh$, this can be written 
\begin{equation}
e^{i\beta_{l,h}} = e^{iG((l+1/2)h)/h}  + O(h) 
\end{equation}
completing the proof of Theorem~\ref{thm:eigenvalue}.

It remains to prove the formula for the Hessian in
\eqref{eq:hessian}.  First we notice that when $\theta = \theta' = \pi/2$ we have, using the formula in \eqref{cosinelaw} for the distance function, in the coordinates $(\tilde \varphi, \theta, \theta')$, 
\begin{equation}
  \label{eq:distancehessian}
 D^2\dist(\varphi, \pi/2, \varphi + \Sigma(\alpha), \pi/2) =  \lp
  \begin{array}{ccc}
    0 & 0 & 0 \\
    0 & \cot \Sigma(\alpha) & - \csc \Sigma(\alpha) \\
    0 &-\csc \Sigma(\alpha)& \cot \Sigma(\alpha)  
      \end{array} \rp
\end{equation}
From this \eqref{eq:phase3d} and \eqref{eq:goodphase}, we conclude that in the $(v,\wt{\varphi}, y, y')$ coordinates
\begin{equation}
  \label{eq:phasehess}
  D^2 \Psi =  \lp \begin{array}{cccc}
  G''(v) & 1 & 0 & 0 \\
  1 & 0 & 0 & 0 \\
    0 &   0 & -\alpha \cot \Sigma(\alpha) + i \alpha &  \alpha
    \csc \Sigma(\alpha) \\
  0 &   0 & \alpha \csc \Sigma(\alpha)& -\alpha \cot
    \Sigma(\alpha)   + i \alpha
          \end{array} \rp . 
\end{equation}
Thus 
$$
\det -i D^{2} \Psi =  -\alpha^2 \Big( \cot^2 \Sigma(\alpha) - 2i \cot \Sigma(\alpha) - 1 - \csc^2 \Sigma(\alpha) \Big) = 
\frac{2i \alpha^{2}}{\sin \Sigma(\alpha)} e^{-i \Sigma(\alpha)}
$$ and \eqref{eq:hessian} holds. 
\end{proof}

\begin{remark}\label{Maslovd=3} The Maslov factor in
  \eqref{eq:princsymbd=3} and \eqref{unitarity} arises as
  follows. First, Lemma~\ref{lem:prsymb1} shows that the Maslov bundle
  over $L$ is canonically trivial. However, unlike in the case $d=2$,
  there is a nontrivial Maslov factor from comparing our phase
  function $\Phi$ above to one --- let us call it $\tilde \Phi$ ---
  that agrees with the canonical phase function, i.e. the
  pseudodifferential phase function, for $|\eta| \geq R$. By
  \cite[Theorem 3.2.1]{H1971}, the principal symbol written relative
  to $\Phi$ contains the Maslov factor $e^{i\pi\sigma/4}$ where
  $\sigma$ is the difference of signatures, 
$$
\sigma = \sgn D^2_{vv} \Phi - \sgn D^2_{\tilde w \tilde w} \tilde \Phi
$$
where $\tilde w = (\tilde w_1, \tilde w_2)$ are the phase variables for $\tilde \Phi$. 
A tedious computation shows that $\sigma = -1$, leading to the Maslov factor in \eqref{eq:princsymbd=3} and \eqref{unitarity}. (We remark that since $\Phi$ depends on one phase variable and $\tilde \Phi$ on two phase variables, by \cite[Equation (3.2.12)]{H1971} $\sigma$ is odd, so the Maslov factor cannot vanish in this case.) 
Of course, the Maslov factors are irrelevant to the question of equidistribution, but they are relevant to the question of determining the eigenvalues modulo $O(h)$. 
\end{remark}

\begin{proof}[Proof of Proposition~\ref{prop:largel}]
In view of the remarks in the proof of
Theorem~\ref{thm:eigenvalue}, specifically equation
\eqref{eq:dimensiondiff}, it is only necessary to do this in the cases
$d= 2$ and $d=3$. For definiteness, we write down the proof for $d=3$;
it is similar, and in fact simpler, for $d=2$. 
Consider a spherical harmonic $Y^l_l$ with $hl \geq R + h^\kappa$,
where $\kappa < 1$. The eigenvalue $e^{i\beta_{l,h}}$ is given by
\eqref{betaK} with $Y_{l}^{l}$ replacing $Y_{l}$.

First assume that $hl \geq R' > R$. Then the $K_1$ 
term in \eqref{betaK} will be $O(h^\infty)$ (for a suitable
decomposition of $S_h = K_1 + K_2 + K_3$ as above, with $R_2 < R'$),
so we only have to consider the $K_2$ term. This is given by a
pseudodifferential operator with symbol equal to $1 + O(h^\infty)$, so
the $\langle Y^l_l, K_2 Y^l_l \rangle$ term is equal to $1 +
O(h^\infty)$, proving the Proposition  in this case.

Next assume that $R + h^\kappa \leq hl \leq R_1'$. For $R' < R_1$, the
$K_2$ term in \eqref{betaK} will be $O(h^\infty)$ (for some other
decomposition of $S_h$, with $R_1 > R'$), so we only need to consider
the $K_1$ term. That is, it remains to show that 
$$
\la (K_1 - \mbox{Id}) Y^{l}_{l},  Y^l_l \ra = O(h^\infty) \mbox{ for } R +
h^{\kappa} \leq hl
$$
Using as above polar coordinates $(\varphi, \theta)$ on $S^2$ 
with dual coordinate $(\sigma, \zeta)$, we find a phase
function $\Psi$ for $K_1$ that parametrizes $L$ microlocally in the region 
$$\{ |(\sigma, \zeta)|_g \geq R - \delta \}$$ for fixed small $\delta > 0$.  Indeed, since $L$ is given by the
diagonal relation
\begin{equation}\label{eq:pseudodiffagain}
\set{ \varphi = \varphi', \, \theta = \theta', \,
\sigma = -\sigma', \, \zeta = -\zeta', \, \tau = 0 } \mbox{ for }
\set{|(\sigma, \zeta)|_g \geq R},
\end{equation}
it follows that the functions
$(\varphi, \theta, \sigma', \zeta')$ furnish local coordinates on the
Legendrian $L$ for $\{ |(\sigma, \zeta)|_g \geq R\}$ and therefore, by
continuity, for $\{ |(\sigma, \zeta)|_g \geq R - \delta\}$ for some
small $\delta > 0$. 
It then follows from \cite[Theorem 21.2.18]{Hvol3} that $L$ can  be parametrized by a phase function of the form
$$
-\varphi' v - \theta' w + H(\varphi, \theta, v, w). 
$$
Since $K_1$ is pseudodifferential for $|(\sigma', \zeta')|^2_g =
{\zeta'}^2 + (\sin \theta')^{-2} {\sigma'}^2 \geq R^2$ (i.e.\ $L$
satisfies \eqref{eq:pseudodiffagain}), we have
\begin{align*}
 v \geq R \sin \theta' \implies  
    H =   \varphi v + \theta w 
    \mbox{ and } 
    b = 1 + O(h^\infty)
\end{align*}
Thus
\begin{equation}\label{K2}
 \begin{split}
\la (K_1 - \mbox{Id}) Y^{l}_{l}, Y^{l}_{l} \ra  =   \int \Bigg( e^{i
  \big(- \varphi' v
  - \theta' w + H(\varphi, \theta, v, w)\big)/h} b(\varphi', \theta',
v, w, h) - e^{i\big( (\varphi - \varphi') v + (\theta - \theta') w
  \big)/h}  \Bigg) \\ \times \, c_l e^{i\alpha (g(\theta) +
  g(\theta'))/h} \, \frac{d\varphi \, d\theta \, d\varphi' \, d\theta' \, dv \, dw}{(2\pi
  h)^{3}} + O(h^{\infty}),
\end{split} 
\end{equation}
As above we have written $l = \alpha/h$; hence $\alpha > R + h^\kappa$.

We insert cutoff functions by writing
$$
1 = \chi \lp \frac{v - R \sin \theta'}{h^\kappa} \rp + (1 - \chi) \lp \frac{v - R \sin \theta'}{h^\kappa} \rp
$$
where $\chi(t)$ is supported in $t \leq 1/2$, equal to $1$ for $t \leq  1/4$. 
With the cutoff $\chi$ inserted, the phase function is nonstationary on the support of the integrand, since stationarity requires that $v = \alpha$. It follows that we can integrate by parts arbitrarily many times, using the fact that the differential operator 
$$
\frac1{v- \alpha} \frac{h}{i} \frac{\partial}{\partial \varphi'}
$$
leaves  both exponential factors invariant; doing this gains a factor of $h^{1 - \kappa}$ each time since $\alpha - v \geq h^\kappa/2$ on the support of the integrand. Thus the $\chi$ term is $O(h^\infty)$. 

With the cutoff $1 - \chi$ inserted, we write the integral in \eqref{K2} in the form 
\begin{equation}
 \begin{split}
 \int  e^{i\big( (\varphi - \varphi') v + (\theta - \theta') w \big)/h}  \Bigg( e^{i\big( H(\varphi, \theta, v, w) - \varphi v - \theta w \big)/h} b(\varphi', \theta', v, w, h) - 1 \Bigg) \\
 \times \, c_l (1 - \chi) \lp \frac{v - R \sin \theta'}{h^\kappa} \rp  e^{i\alpha (g(\theta) +
  g(\theta'))/h} \, \frac{d\varphi \, d\theta \, d\varphi' \, d\theta' \, dv \, dw}{(2\pi
  h)^{3}}.
\end{split} 
\end{equation}
We claim that the factor 
$$
\Bigg( e^{i\big( H(\varphi, \theta, v, w) - \varphi v - \theta w \big)/h} b(\varphi', \theta', v, w, h) - 1 \Bigg) \times  (1 - \chi) \lp \frac{v - R \sin \theta'}{h^\kappa} \rp 
$$
is $O(h^\infty)$. In fact, the term in the large brackets is $O(h^\infty)$ for $v \geq R \sin \theta'$, while if  $v \leq R \sin \theta'$, then the $1 - \chi$ term vanishes identically.  It follows that the $1 - \chi$ term is also $O(h^\infty)$, completing the proof of Proposition~\ref{prop:largel}. 
\end{proof}

\begin{remark}\label{FIOvsODE} The reader may wonder whether a direct
  ODE attack on \eqref{eq:1} might be simpler and more straightforward
  than our FIO approach to this problem, given that our approach
  relies on \cite[Theorem 15.6]{HW2008}, which in turn rests on a
  significant amount of machinery. By contrast the WKB expansion for
  the solution yields the approximation for the eigenvalues in
  Theorem~\ref{thm:eigenvalue} in a straightforward fashion. However,
  although it is not hard to write down a WKB approximation to the
  solutions of \eqref{eq:1}, 
it seems (to the authors) that proving rigorous error bounds for such
WKB expansions is rather subtle. The problem is that to prove such
bounds, one must solve away the error term, that is,  get good
estimates on the solution to the imhomogeneous ODE where the
inhomogeneous term (the error term when the WKB approximation is
substituted into \eqref{eq:1}) is $O(h^N)$ for some sufficiently large
$N$. Notice that the ODE \eqref{eq:1} might have several turning
points, and the desired solution is governed by a boundary condition
$f'(0) = 0$ at the origin, so one needs to understand the behaviour of
the solution passing through possibly several turning points. Since
the solutions may grow exponentially in the non-interaction region,
this does not seem to be easy or straightforward, and we are not aware
of anywhere in the literature where this has been written
down. Carrying out this procedure would certainly be a worthwhile
enterprise, but we have chosen instead to build on the above-mentioned
theorem about the semiclassical scattering matrix which is already
available in the literature.

Other features recommend the FIO approach in this context.  First, the
relationship between the scattering angle and the phase-shifts is
made transparent here, or at least it is `reduced' to the fact that the
integral kernel of $S_{h}$ is a semi-classical FIO whose canonical
relation `contains' the scattering angle, while on the other hand from
the formula produced by the WKB this relationship is not immediately
apparent.  More importantly, FIO methods will be essential in treating
the noncentral case, which we intend to do in future work, and the
symmetric case under consideration is a situation in which $S_{h}$ can 
be understood almost explicitly.

\end{remark}


\section{Equidistribution}\label{sec:equid}

If $\omega = \set{e^{2\pi i x_{1}}, \dots, e^{2\pi i x_{K}}}$ is any set of $K$ points on
  $\mathbb{S}^{1}$, then the \textbf{discrepancy} $D(\omega)$ is
  defined by
  \begin{equation}
    \label{eq:discrepancy}
    D(\omega) :=  \sup_{0 \le \phi_{0} < \phi_{1} \le 2 \pi}
    \absv{\frac{N(\phi_{0}, \phi_{1} ; \omega)}{K} -
    \frac{\phi_{1} - \phi_{0}}{2\pi}},
  \end{equation}
where $N(\phi_{0}, \phi_{1} ; \omega)$ is the number of points in
$\omega$ with argument in $[\phi_{0}, \phi_{1}]$ (modulo
$2\pi$), counted with multiplicity.  We state the following lemma in
slightly more generality then is
necessary for semiclassical potentials so that we may apply it
without significant modification to the case of scattering by the disk.

\begin{lemma}\label{thm:exparg}
  Let $G \colon [0, R)  \lra \mathbb{R}$ be 
  smooth and assume that 
\begin{equation}
  \set{ \alpha :
    G''(\alpha) = 0} \mbox{ is finite in } [0, R).
\label{G''condition}\end{equation}    
 Consider the points $x_{lk}$ on the unit circle
 \begin{equation}
    \label{eq:Eh}
   \mathcal{E}_{h} = \set{ x_{lk}  :=  \exp{( i G(lh)/h )}  :  0 \le lh
     < R, k
   = 1, \dots, p_{d}(l)},
  \end{equation}
included according to multiplicity.  Here
$p_{d}(l) = \dim \ker \lp \Delta_{\mathbb{S}^{d - 1}} - l(l + d
- 2) \rp$.

Then the sets $\mathcal{E}_{h}$ equidistribute as $h \to
0$.  That is, the discrepancy satisfies
\begin{equation}
  \label{eq:equidistributed}
    \lim_{h \to 0} D(\mathcal{E}_{h}) = 0.
\end{equation}
\end{lemma}

To apply the lemma to the eigenvalues of the scattering matrix
$S_{h}$, we must show that they still equidistribute despite
satisfying only the weaker asymptotic condition in Theorem
\ref{thm:eigenvalue}.  

\begin{proposition}\label{thm:realrealdiscrep}
Let $S \subset [0,R]$ be a finite set and let  
   \begin{equation}
    \label{eq:tildeEh}
   \widetilde{\mathcal{E}}_{h} = \set{ \wt{x}_{lk}  :  0\le lh
     \le R, k
   = 1, \dots, p_{d}(l)}
  \end{equation}
be a collection of points on $\mathbb{S}^{1}$ (included according to
multiplicity), such that for any $\epsilon > 0$, if $l$ satisfies $\dist(lh, S)\ge \epsilon$
then $$\wt{x}_{lk}  =
\exp{ i( G(lh)/h) +  E(l,h))}$$ where $\absv{E(l,h)} <
C(\epsilon)h$.
Then, if $G$ satisfies condition \eqref{G''condition} in Lemma \ref{thm:exparg}, $$ \lim_{h \to 0}
D(\wt{\mathcal{E}_{h}}) = 0.$$
\end{proposition}
We will use the following notation.  With any set $S$ as above, let
\begin{equation}
  \label{eq:between}
  \begin{split}
    \mathcal{E}_h(\epsilon) := \mathcal{E}_h \cap \{x_{lk}\colon \dist(lh,S)\ge
    \epsilon\} \quad \mbox{and} \quad
    \wt{\mathcal{E}}_h(\epsilon) := \wt{\mathcal{E}}_h \cap
    \{\wt{x}_{lk}\colon \dist(lh,S)\ge \epsilon\},
  \end{split}
\end{equation}
always understood to include points according to multiplicity.  
\begin{proof}[Proof of Proposition \ref{thm:realrealdiscrep} assuming
  Lemma \ref{thm:exparg}:]
 The error bound $\absv{E(l,h)} <
C(\epsilon)h$ shows that,  for every $\epsilon > 0$,
there is a constant $C = C(\epsilon, S) > 0$ so
that 
\begin{equation*}
  N(\phi_{0} + Ch, \phi_{1} - Ch; \mathcal{E}_{h}(\epsilon))  \le N(\phi_{0},
  \phi_{1} ; \wt{\mathcal{E}_{h}}(\epsilon)) \le N\big(\max(\phi_{0} - Ch,0),
  \min(\phi_{1} + Ch, 2\pi); \mathcal{E}_{h}(\epsilon)\big).
\end{equation*}
Dividing through by $2 (R/h)^{d - 1} / (d - 1)!$, subtracting $(\phi_{1} -
\phi_{0})/2\pi$, and taking $h$ small gives
\begin{align*}
\absv{\frac{N(\phi_{0}, \phi_{1};  \wt{\mathcal{E}_{h}}(\epsilon))}{
   2 (R/h)^{d - 1} / (d - 1)! }  - \frac{\phi_{1} -
\phi_{0}}{2\pi} } & \le \max \left\{
   \absv{\frac{N(\phi_{0}+ Ch, \phi_{1} - Ch;  \mathcal{E}_{h}(\epsilon))}{
   2 (R/h)^{d - 1} / (d - 1)! }  - \frac{\phi_{1} -
\phi_{0}}{2\pi} }, \right. \\
& \qquad \qquad \left. \absv{\frac{N(\phi_{0}- Ch, \phi_{1} + Ch; \mathcal{E}_{h}(\epsilon))}{
   2 (R/h)^{d - 1} / (d - 1)! } -  \frac{\phi_{1} -
\phi_{0}}{2\pi} }
 \right\} \\
& \le D(\mathcal{E}_{h}(\epsilon)) + (1 + C(\epsilon))O(h) + O(\epsilon),
\end{align*}
uniformly in $h$ and $\epsilon$, where for the second inequality we used
\begin{equation}
  \label{eq:8}
  \absv{\mathcal{E}_{h}(\epsilon)} = \frac{2R^{d-1}}{h^{d-1}(d-1)!}\lp 1 + O(h)
    + O(\epsilon) \rp,
\end{equation}
where $\absv{\mathcal{E}_{h}(\epsilon)}$ is the number of points in
$\mathcal{E}_{h}(\epsilon)$. 
Similarly, $D(\mathcal{E}_{h}) = D(\mathcal{E}_{h}(\epsilon)) + O(h) +
O(\epsilon)$ for $h, \epsilon$ small, and the same is true
for $\wt{\mathcal{E}}_{h}$.  Thus
\begin{align}\label{eq:cepsilon}
  \absv{\frac{N(\phi_{0}, \phi_{1}; \wt{\mathcal{E}_{h}})}{
   2 (R/h)^{d - 1} / (d - 1)! }  - \frac{\phi_{1} -
\phi_{0}}{2\pi} } &\le D(\mathcal{E}_{h})  +  (1 + C(\epsilon))O(h) + O(\epsilon).
\end{align}
Thus
\begin{equation*}
	\limsup_{h \to 0} D(\wt{\mathcal{E}}_{h}) = O(\epsilon),
\end{equation*}
and as $\epsilon > 0$ was arbitrary, we obtain the result.
\end{proof}

\begin{remark}\label{thm:dobetter}
  Note that the proof gives no information about the exact vanishing
  rate of $D(\wt{\mathcal{E}}_{h})$ as $h \to 0$.  For this, one must
  have information on the dependence of $C(\epsilon)$ on $\epsilon$,
  and then 
  optimize in $\epsilon$ in \eqref{eq:cepsilon} as $h \to 0$.  This is what we
do in Section \ref{sec:disk} to obtain improved remainders in the case
of scattering by the disk.
\end{remark}

To prove Lemma~\ref{thm:exparg}, we use theorems from \cite{KN1974}.
The following theorem follows from \cite[ch.\ 2, eq.\ 2.42]{KN1974}:
\begin{theorem}[Erd\"os-Tur\'an]\label{thm:erdosturan}
  There is a constant $c > 0$ such that if $$\omega = \set{e^{2\pi i x_{1}}, \dots, e^{2\pi i x_{N}}}$$ is a finite sequence of $N$ points on
  $\mathbb{S}^{1}$ and $m$ is any positive integer, then
  \begin{equation}
    \label{eq:ET}
    D(\omega) \le c \lp
    \frac{1}{m} + \sum_{j = 1}^{m} \frac{1}{j} \absv{
      \frac{1}{N}\sum_{l = 1}^{N} e^{2 \pi i j x_{l} }}
    \rp.
  \end{equation}
\end{theorem}
To bound the exponential sums that appear on the right hand side of
\eqref{eq:ET}, we use
\cite[ch.\ 1, thm.\ 2.7]{KN1974}, namely
\begin{theorem}\label{thm:expsum}
  Let $a$ and $b$ be integers with $a < b$, and let $f$ be twice
  differentiable on $[a, b]$ with $\absv{f''(x)} \ge \rho > 0$ for $x \in [a, b]$.  Then 
  \begin{equation}
    \label{eq:expsumbound}
    \absv{ \sum_{l = a}^{b} e^{2\pi i f(l)} } \le \lp \absv{ f'(b) - f'(a) }
    + 2 \rp \lp \frac{4}{\sqrt{\rho}} + 3 \rp . 
  \end{equation}
\end{theorem}
We also need \cite[thm.\ 2.6]{KN1974} (with minor
modifications in notation):
\begin{theorem}\label{thm:superposition}
  For $1 \le i \le k$, let $\omega_{i}$ be a set of $\absv{\omega_{i}}$
  points on $\mathbb{S}^{1}$ with discrepancy $D(\omega_{i})$.  Let
  $\omega$ be a concatenation of $\omega_{1}, \dots, \omega_{k}$, that
  is, a set obtained by listing in some order the terms of the
  $\omega_{i}$.  Then 
  \begin{equation}
    \label{eq:superdisc}
    D(\omega) \le \sum_{i = 1}^{k} \frac{\absv{\omega_{i}}}{\absv{\omega}} D(\omega_{i}),
  \end{equation}
where $\absv{\omega}$ is the number of points in $\omega$.
\end{theorem}

\medskip 

\begin{proof}[Proof of Lemma \ref{thm:exparg}] We begin by assuming
  that $G''$ has no zeroes in the
  open interval $(0, R)$.  

We first
  analyze the subset $\mathcal{E}_{h}(\epsilon) \subset
  \mathcal{E}_{h}$ defined in \eqref{eq:between}.
Define
\begin{equation}
  \label{eq:rhotilde}
  \begin{split}
    \tilde{\rho} = \tilde{\rho}(\epsilon) &= \min_{\epsilon \le \alpha \le R -
      \epsilon} \absv{G''(\alpha) }\\
    \tilde{\kappa} =   \tilde{\kappa}(\epsilon) &= 2 \max_{\epsilon \le \alpha \le R -
        \epsilon} \absv{G'(\alpha)}.
  \end{split}
\end{equation}
We will show that for each
$\gamma \in (0, 1)  $ there is a constant $c =
c(\gamma) > 0$ so that for each $\epsilon > 0$, 
\begin{equation}\label{eq:epsilondiscrep}
\begin{split}
   D(\mathcal{E}_{h}(\epsilon)) 
& <  c \lp h^{\gamma} + \tilde{\kappa} \tilde{\rho}^{-1/2} h^{1/2 -\gamma/2} + \tilde{\rho}^{-1/2} h^{1/2} + \tilde{\kappa} h^{1 - \gamma} \rp . 
\end{split}
\end{equation}
Since $\mathcal{E}_{h}(\epsilon) - \mathcal{E}_{h} = h^{-d+ 1}(O(\epsilon)
+ O(h))$, for some $c = c(\gamma) > 0$ independent of $\epsilon$ we
have
\begin{equation}\label{eq:realdiscrep}
\begin{split}
D(\mathcal{E}_{h}) & = \sup_{0 \le \phi_{0} < \phi_{1} \le 2\pi}  \absv{\frac{N(\phi_{0}, \phi_{1}; \mathcal{E}_{h})}{\absv{\mathcal{E}_{h}}} - \frac{\phi_{1} -
   \phi_{0}}{2\pi}}\\
   &\le c(\epsilon + h) + D(\mathcal{E}_{h}(\epsilon)),
  \end{split}
\end{equation}
showing that
\begin{equation*}
\limsup_{h \to 0} D(\mathcal{E}_{h}) \le c \epsilon. 
\end{equation*}
Since $\epsilon > 0$ is arbitrary, this gives \eqref{eq:equidistributed}. Thus it remains to prove \eqref{eq:epsilondiscrep}. 

\vskip 5pt

\noindent \textit{Case 1: dimension $d = 2$.} 
Note that when $d = 2$ the multiplicity of the eigenspaces is $p_{2}(l) = 1$ if $l
= 0$ and $2$ otherwise, so that 
\begin{equation*}
  \absv{\mathcal{E}_{h}(\epsilon)} = 2
\lp \lfloor (R - \epsilon)/h \rfloor - \lceil \epsilon/h \rceil + 1
\rp.
\end{equation*}
  We apply Theorem
\ref{thm:erdosturan} with $\omega = \mathcal{E}_{h}(\epsilon)$, so
that, in the notation of Theorem \ref{thm:erdosturan},
  $x_{l} = G(lh)/(2\pi h)$.  Thus
\begin{equation*}
  D(\mathcal{E}_{h}(\epsilon)) \le c \lp \frac{1}{m} + \sum_{j =
    1}^{m} \frac{1}{j} \absv{ \frac{1}{\lfloor (R - \epsilon)/h \rfloor -
      \lceil \epsilon/h \rceil + 1} \sum_{l = \lceil \epsilon/h \rceil}^{\lfloor (R - \epsilon) / h \rfloor} e^{
      i j G(lh) / h } }
  \rp
\end{equation*}
Then we apply Theorem \ref{thm:expsum} with $f(x) = (j/2\pi)G(xh)/h$, $a = \lceil \epsilon/h \rceil$, and $b =
  \lfloor (R - \epsilon) / h \rfloor$.
Thus, if $xh \le R -
  \epsilon$ then $\absv{f''(x)} = h j \absv{G''(xh)} / 2 \pi \ge h j \tilde{\rho} / 2 \pi$, which equals $\rho$ in the notation of Theorem \ref{thm:expsum}.  
It follows that
\begin{equation}    
\label{eq:ET2}
D(\mathcal{E}_{h}(\epsilon)) \le c \lp \frac{1}{m} + \frac{h}{R - 2\epsilon - h} \sum_{j = 1}^{m} \frac{1}{j} \lp
    \frac{j \tilde{\kappa}}{\pi} + 2 \rp \lp \lp \frac{32 \pi}{h j \tilde{\rho}} \rp^{1/2} + 3\rp \rp.
\end{equation}
By letting $m = \lfloor h^{-\gamma} \rfloor$ for any $\gamma > 0$, we obtain \eqref{eq:epsilondiscrep}.

Finally, suppose there are a finite number of points  $0 < a_1 <
\ldots < a_{n - 1} < R$ with $G''(\alpha_{i}) = 0$, and let $a_0 = 0,
a_{n} = R$.  Note
that, if we define $\mathcal{E}_{h}(a, b)$ to be the set of
$x_{l,k}$ with $a \le lh \le b$, counted with multiplicity, then the above
arguments show that $\lim_{h \to 0}D(\mathcal{E}_{h}(a,b)) = 0$; in fact if
$\wt{\rho}(\epsilon, h)$ (resp.\ $\wt{\kappa}$) is defined to be
$\min_{a + \epsilon \le \alpha \le b - \epsilon}\absv{G''(\alpha)}$
(resp.\ $\max \absv{G'(\alpha)}$), then the proof is the same.  The lemma in the $d = 2$
case now follows from Theorem~\ref{thm:superposition} since by \eqref{eq:superdisc}
\begin{equation*}
  D(\mathcal{E}_{h}) \le \sum_{i = 1}^{n} D(\mathcal{E}_{h}(a_{i - 1}, a_{i})) .
\end{equation*}
The proof is now complete in the case $d = 2$.

\vskip 5pt

\noindent \textit{Case 2: dimension $d > 2$.}  As in the $d = 2$ case, we begin by assuming that $G''(\alpha)$ has
  zeroes only at $0$ and $R$. We now have to deal with the increasing multiplicities $p_d(l)$.

We will apply Theorem \ref{thm:superposition} to
$D(\mathcal{E}_{h}(\epsilon))$ decomposed as a superposition in the
following way.  It will be convenient to set
\begin{equation}
  \label{eq:N}
  N := \lfloor (R - \epsilon)/h \rfloor.
\end{equation}
Define, 
 \begin{align*}
   \omega(n) : = &  \set{ e^{i G(l h)/h} :  n \le l  \le N
     \rfloor}
 \end{align*}
with \textit{unit multiplicity}.
Note that $\omega(n)$ has $N - n + 1$ elements.
Setting
\begin{align*}
  \omega_{1} &= \omega(0) \\
  \omega_{2} &= \dots = \omega_{p_d(1)}  = \omega(1) \\
  \omega_{p_d(1) + 1} &= \dots = \omega_{p_d(2)} = \omega(2) \\
  \vdots &  \\
  \omega_{p_{d}(N - 1) + 1 }&= \dots = \omega_{p_{d}(N)} = \omega(N).
\end{align*}
we see that the set $\mathcal{E}_{h}(\epsilon)$ is the superposition of the sets $\omega_{1},
    \dots \omega_{p_{d}(N)}.$

The discrepancy $D(\omega(n))$ can be estimated using the method from
the $d = 2$ case.  In particular, as in \eqref{eq:ET2} we see that for
any positive integer $m$, 
\begin{equation}\label{eq:dwn}
 D(\omega(n)) \le c \lp \frac{1}{m} + \sum_{j = 1}^{m} \frac{1}{(N - n + 1) j}\absv{ \lp
    \tilde{\kappa}j + 2 \rp \lp \lp\frac{32}{j c \tilde{\rho} h}\rp^{1/2} + 3\rp }\rp
\end{equation}
By Theorem \ref{thm:superposition}, we have
\begin{equation}
  \label{eq:superdiscd}
  \begin{split}
    D(\mathcal{E}_{h}(\epsilon)) &\le \sum_{i = 0}^{p_{d}(N)}
    \frac{\absv{\omega_{i}}}{\absv{\mathcal{E}_{h}(\epsilon)}}
    D(\omega_{i}) \\
    &\le \frac{1}{\absv{\mathcal{E}_{h}(\epsilon)}} \sum_{n = 0}^{N}
    \absv{\omega(n)}
   (p_{d}(n) - p_{d}(n - 1))  D(\omega(n)) \\
   &\le ch^{d - 1} \sum_{n = 1}^{N}
    (N - n + 1) (n + 1)^{d - 3}  D(\omega(n)),
  \end{split}
\end{equation}
Substituting the estimate \eqref{eq:dwn} into
\eqref{eq:superdiscd}, again
with $m = \lfloor h^{-\gamma} \rfloor$ for some fixed $\gamma \in (0,
1)$, we end up with five terms to deal with corresponding to the five
terms in the right hand side of \eqref{eq:dwn}.  For
all of these we use standard bounds for sums of polynomials and $N
\sim c/h$. The
easiest is the $1/m$ term, since
\begin{equation}
  \label{eq:discrepmult1}
  h^{d - 1}\sum_{n = 1}^{N}
    (N - n + 1) (n + 1)^{d - 3} h^{\gamma} \le c h^{\gamma}.
\end{equation}
Next we do the terms involving $\tilde{\rho}$.  There is
\begin{equation}
  \label{eq:discrepmult2}
  \begin{split}
  &   h^{d-1}\sum_{n = 1}^{N} (N - n + 1) (n + 1)^{d - 3}\sum_{j = 1}^{\lfloor h^{-\gamma} \rfloor}
    \frac{1}{(N - n + 1)j}\lp
    \tilde{\kappa}j\rp \lp \frac{32}{j c \tilde{\rho} h}\rp^{1/2}  \\
  &\le c h^{d-1} \tilde{\kappa} \sum_{n = 1}^{N}  (n + 1)^{d - 3}\sum_{j = 1}^{\lfloor h^{-\gamma} \rfloor}
        \lp \frac{32}{j \tilde{\rho} h}\rp^{1/2} \\
  & \le  c h^{d-1} \tilde{\kappa}\lp \frac{1}{\tilde{\rho} h}\rp^{1/2} \sum_{n = 1}^{N}  (n + 1)^{d - 3}\sum_{j = 1}^{\lfloor h^{-\gamma} \rfloor}
        j^{-1/2} \\
  & \le  c \lp \frac{1}{\tilde{\rho} }\rp^{1/2} \tilde{\kappa} h^{1/2 - \gamma/2},
  \end{split}
\end{equation}
and
\begin{equation}
  \label{eq:discrepmult3}
  \begin{split}
  &   h^{d-1} \sum_{n = 1}^{N} (N - n + 1) (n + 1)^{d - 3}\sum_{j = 1}^{\lfloor h^{-\gamma} \rfloor}
    \frac{2}{(N - n + 1)j}\lp \frac{32}{j c \tilde{\rho} h}\rp^{1/2}  \\
  & \le  h^{d-1} c \lp \frac{1}{\tilde{\rho}h}\rp^{1/2} \sum_{n = 1}^{N}  (n + 1)^{d - 3}\sum_{j =
    1}^{\lfloor h^{-\gamma} \rfloor}
        j^{-3/2} \\
  & \le  c \lp \frac{1}{\tilde{\rho} }\rp^{1/2} h^{-1/2}.
  \end{split}
\end{equation}
The other terms are
\begin{equation}
  \label{eq:discrepmult4}
  \begin{split}
  &   h^{d-1} \sum_{n = 1}^{N} (N - n + 1) (n + 1)^{d - 3}\sum_{j = 1}^{\lfloor h^{-\gamma} \rfloor}
    \frac{1}{(N - n + 1)j}\lp
    \tilde{\kappa}j\rp \times 3  \\
  &\le  c \tilde{\kappa} h^{1 - \gamma} \le  c \tilde{\kappa} h^{1/2 - \gamma/2},
  \end{split}
\end{equation}

and
\begin{equation}
  \label{eq:discrepmult5}
  \begin{split}
  &   h^{d-1} \sum_{n = 1}^{N} (N - n + 1) (n + 1)^{d - 3}\sum_{j = 1}^{\lfloor h^{-\gamma} \rfloor}
    \frac{6}{(N - n + 1)j}   \le c h \log (1/h). 
  \end{split}
\end{equation}
Combining \eqref{eq:discrepmult1}--~\eqref{eq:discrepmult5} with
\eqref{eq:superdiscd} gives \eqref{eq:epsilondiscrep}.

We take care of
the case of a non-trivial number of zeroes of $G''$ on $[0, R]$
exactly as in the $d = 2$ case. This completes the proof of Lemma \ref{thm:exparg}. 

\end{proof}

\bigskip

We can now prove Theorem \ref{thm:main}.
\begin{proof}[Proof of Theorem \ref{thm:main}.]
By Theorem \ref{thm:eigenvalue}, the eigenvalues of the scattering matrix for $0 \leq l \leq R/h$ are given by 
$$
\exp\set{ \frac{i}{h} \lp G \big((l + \frac{d-2}{2})h \big)  \rp }+ O(h)
$$
in the case of even dimension $d$, and the same in odd dimensions away from any $\epsilon$ neighbourhood of the set  $S = \{ \alpha \mid G'(\alpha)/\pi \in \mathbb{Z} \}$, where $\alpha = l/h$. Since by assumption $\Sigma = G'$ satisfies \eqref{scattering-angle-condition}, $G$ satisfies  \eqref{G''condition}, 
so the conclusion of Lemma~\ref{thm:exparg} holds for $G$. Finally, \eqref{G''condition} implies that $S$ is a finite set, so we can apply Proposition \ref{thm:realrealdiscrep}, proving \eqref{eq:maintheorem} and hence Theorem \ref{thm:main}. 
\end{proof}

\bigskip


\section{Examples of potentials that satisfy Assumption \ref{scattering-angle-condition}
  }\label{sec:example}  
  We use expression \eqref{eq:integral} for the scattering angle to prove 

\begin{proposition}\label{thm:examplepots}
  Suppose that on the region of interaction $\mathcal{R}$ the potential $V$ satisfies
 \begin{equation}
V'(r) \leq 0  \text{ and }
  (V')^{2} + (1 -V) (V' + rV'') > 0 \text{ for } r < R. 
\label{Vconditions}\end{equation}
  Then
  $\Sigma'(\alpha) <
  0$ for for $\alpha \in [0, R)$.  
  
The conditions in \eqref{Vconditions} hold in particular for $V = cW$, where $c$ is
  sufficiently large and where $W(r) = 0$ for $r \geq R$, $W(r)
  > 0$ for $0 \leq r < R$ and $W''(r)$ is positive and monotone
  decreasing in some nonempty interval $[R- \epsilon, R)$. 
  An explicit example is 
$$
W(r) = \begin{cases} e^{1/(r^2 - R^2)}, \quad \, r < R \\
0, \qquad \qquad \quad r \geq R. \end{cases}
$$
\end{proposition}

\begin{proof}
In \eqref{eq:integral}, set $s = r / r_{m}$, so
\begin{align*}
    \Sigma(\alpha)  &= \pi -  2 \int_{r_{m}}^{\infty} \frac{\alpha}{(sr_{m})^{2} \sqrt{1
      - \alpha^{2}/((sr_{m})^{2} - V(sr_{m})}} \, d(sr_{m}) \\
  &=\pi -  2 \int_{1}^{\infty} \frac{\alpha}{s^{2} \sqrt{r_{m}^{2}
      - \alpha^{2}/s^{2} - r_{m}^{2}V(sr_{m})}} \, ds.
\end{align*}
Differentiating under the integral sign gives
\begin{align*}
- \frac{1}{2}   \Sigma'(\alpha) &= \int_{1}^{\infty} 
  \lp \frac{ 1 }{s^{2} \sqrt{ r_{m}^{2}
      - \alpha^{2}/s^{2} - r_{m}^{2}V(sr_{m})}} \right. \\
    & \quad \left. - \frac{\alpha}{2}  \  \frac{  
      \ 2 r_{m}r_m' - 2\alpha/s^{2}  -
      2r_{m} V(sr_{m}) r_m'  - r_{m}^{2} s
      V'(sr_{m}) r_m' }{s^{2} \lp r_{m}^{2}
      - \alpha^{2}/s^{2} - r_{m}^{2}V(sr_{m})\rp^{3/2}} \rp ds\\
     &= \int_{1}^{\infty} 
  \frac{\lp r_{m}^{2}( 1 - V(sr_{m})\rp -  \alpha r_{m}
    r_m' (1  - V(s r_{m}) -  \frac{1}{2}  r_{m} s
      V'(sr_{m}) )}{s^{2} \lp r_{m}^{2}
      - \alpha^{2}/s^{2} - r_{m}^{2}V(sr_{m})\rp^{3/2}} \, ds
\end{align*}
Differentiating
\begin{equation}
1 - \alpha^{2}/r_{m}^{2} - V(r_{m}) = 0\label{eq:rmin}
\end{equation}
shows $\alpha r_{m}r_m' = \alpha^{2} \big( 1 -
    V(r_{m}) -  \frac{1}{2}r_{m} V'(r_{m}) \big)^{-1}.$
Plugging this in gives
\begin{align*}
-\frac{1}{2}  \Sigma'(\alpha) &= \int_{1}^{\infty} 
 \lp r_{m}^{2}( 1 - V(sr_{m})) - \alpha^{2} \frac{1  - V(s
      r_{m}) -  \frac{1}{2} r_{m} s  V'(sr_{m}) }{1 - V(r_{m}) -
      \frac{1}{2} r_{m} V'(r_{m}) } \rp    \\  & \quad \times  \frac{1}{s^{2} \lp r_{m}^{2}
      - \alpha^{2}/s^{2} - r_{m}^{2}V(sr_{m})\rp^{3/2}}  \, ds
\end{align*}
and using \eqref{eq:rmin} again shows that $(1/2) 
\Sigma'(\alpha)$ is equal to
\begin{align*}
  \label{eq:angularvariation}
 & \int_{1}^{\infty} 
  \frac{  1  - V(s
      r_{m}) -  \frac{1}{2} r_{m} s  V'(sr_{m} ) }  {s^{2} \lp r_{m}^{2}
      - \alpha^{2}/s^{2} - r_{m}^{2}V(sr_{m})\rp^{3/2}}  
      \\ & \qquad \times 
   \lp
    \frac{1 - V(sr_{m})}{1  - V(s
      r_{m}) -  \frac{1}{2} r_{m} s  V'(sr_{m})} - \frac{1 - V(r_{m}) }{ 1 - V(r_{m}) -
      \frac{1}{2} r_{m} V'(r_{m}) }
   \rp r_{m}^{2}  ds \\
 & = \int_{r_{m}}^{\infty} 
  \frac{ 1  - V(r) -  \frac{1}{2} r  V'(r )}{r^{2} \lp 1
      - \alpha^{2}/r^{2} - V(r)\rp^{3/2}} 
     \lp 
   \frac{1 - V(r)}{1  - V(r) -  \frac{1}{2} r  V'(r)} - \frac{1 - V(r_{m}) }{ 1 - V(r_{m}) -
      \frac{1}{2} r_{m} V'(r_{m}) }  \rp dr .
\end{align*}
Differentiating the expression $(1 - V(r)) / (1  - V(r) -  \frac{1}{2}
r  V'(r)) $ with respect to $r$ and using $V' \le 0$, we see that the integrand is
positive if for $r_{m} < r < R$ if
\begin{equation}\label{eq:wierdcondition}
 r (V')^{2} + (1 -V) (V' + rV'') > 0.
\end{equation}
These conditions are implied by \eqref{Vconditions} in the first paragraph of the proposition. 

To check that the condition in the second paragraph is 
sufficient, observe the following.  If
$V(r) \ge 0$, $V'(r) < 0$ and $V''(r) > 0$ on some open interval $(R -
\epsilon, R)$, it follows that for $r$ sufficiently close to $R$, that $ V' +
rV'' > 0$.  By picking large enough $c > 0$, 
\eqref{eq:wierdcondition} will hold on the region of interaction
$\mathcal{R}$.
\end{proof}

\begin{remark}
One can ask whether there exist potentials for which equidistribution fails. 
It is clear from Theorem~\ref{thm:eigenvalue} that the scattering matrix for $V$ will
fail to be equidistributed if the scattering angle $\Sigma(\alpha)$ associated to $V$ is
equal to a constant rational multiple of $2\pi$ on some interval with $\alpha < R$. 
So we can ask whether there exists such a potential. Let $\mathcal{S}$ be the map \eqref{eq:integral} taking $V$ to its
scattering angle $\Sigma$. Linearizing $\mathcal{S}$ 
at the zero potential gives an integral operator which is an elliptic pseudodifferential
operator of order $1/2$ (apart from an extra singularity at $r=\alpha = 0$). This makes
it seem likely to the authors that the range of $\mathcal{S}$ is quite large, very likely
including scattering angles such as described above that would imply non-equidistribution. 
\end{remark}


\section{Scattering by the disk}\label{sec:disk}

In this section we will prove Theorem \ref{thm:disk} from the
introduction.  We restrict our attention to the ball of radius $1$,
since the phase shifts for the ball of radius $R$ can be obtained from those for
$R = 1$ by a scaling argument.

Here we use an ODE analogous to that in \eqref{eq:1} to give a
formula for the eigenvalues.  In fact, for any smooth solution $f_{l}$ to $\Delta_{\mathbb{S}^{d - 1}} f_{l} =
l(l + d - 2) f_{l}$, a straightforward computation shows that
\begin{equation}\label{eq:smatrix}
  S_{k}(f_{l}) = - \frac{H^{(1)}_{l + (d - 2)/2}(k)}{H^{(2)}_{l + (d -
      2)/2}(k)} f_{l},
\end{equation}
where the $H^{(i)}_{\nu}$ are Hankel functions
of order $\nu$ \cite{AS1964}. It follows that
\begin{equation}
S_{k}(f_{l}) = e^{ix_{k,l}} f_l, \quad x_{k,l} = 2 \arg H^{(1)}_{l +
  (d - 2)/2}(k) + \pi .
\label{arg}\end{equation}

We now prove the first part of Theorem~\ref{thm:disk}, which amounts to determining the asymptotics of the argument of the Hankel function when $l/k \leq 1 - k^{-1/3}$. Let us define in this section $\nu = l + (d-2)/2$ and $\alpha = \nu/k$, and study the range $\alpha \leq 1 - k^{-1/3}$.

We first consider the range where $\alpha$ is small, say $\alpha \leq 3/4$. Then we use the expressions \cite[9.1.22]{AS1964} for $J_\nu$ and $Y_\nu$ to derive 
$$
H^{(1)}_{\alpha k}(k) = \frac1{\pi} \int_0^\pi e^{ik (\sin \theta - \alpha \theta)} \, d\theta - \Big( \text{ integrals from $0$ to $\infty$} \Big) .
$$
It is easy to bound the integrals from $0$ to $\infty$ by $O(k^{-1})$ uniformly for $\alpha \leq 3/4$. On the other hand, stationary phase applied to the $\theta$ integral gives 
$$
H^{(1)}_{\alpha k}(k) = \sqrt{\frac{2}{\pi k}} (1 - \alpha^2)^{-1/4} e^{ik(\sqrt{1 - \alpha^2} - \alpha \cos^{-1}\alpha)} e^{-i\pi/4} + O(k^{-1}),
$$
from which it follows that 
\begin{equation}
2 \arg H^{(1)}_{\alpha k}(k) = 2k (\sqrt{1 - \alpha^2} - \alpha \cos^{-1}\alpha) - \pi/2 + O(k^{-1/2})
\label{smallalpha}\end{equation}
in this range. In view of \eqref{arg} and \eqref{Gbdef} this proves \eqref{ball-eval-approx} in the range $\alpha \leq 3/4$. 

In the range $1/2 \leq \alpha \leq 1 - k^{-1/3}$, we use the asymptotic formulae \cite[9.3.35, 9.3.36]{AS1964} which shows that 
\begin{equation}
H^{(1)}_{\alpha k}(k) = \Big( \frac{-4\zeta}{\alpha^{-2} - 1} \Big)^{1/4} \nu^{-1/3} \Big( (\Ai - i \Bi)(\nu^{2/3} \zeta) \Big) (1 + O(k^{-2})) + O(k^{-3/2}), \quad \nu = \alpha k, 
\label{H-asympt}\end{equation}
where $\zeta = \zeta(\alpha)$ is defined by 
\begin{equation}
 \frac{2}{3} \lp - \zeta \rp^{3/2} = \int_{1}^{\alpha^{-1}} \frac{\sqrt{t^{2} -
        1}}{t} dt = \sqrt{\alpha^{-2} - 1} - \cos^{-1}(\alpha) ;
\label{zeta-alpha}\end{equation}
 notice that $\zeta$ is real and negative for $\alpha < 1$, and $-\zeta \sim c (1 - \alpha)$ for some positive $c$ as $\alpha \to 1$. To derive \eqref{H-asympt} from \cite[9.3.35, 9.3.36]{AS1964} we used the fact that $\zeta$ lies in a compact set in this range of $\alpha$, that the $a_k$ and $b_k$ are therefore uniformly bounded, that $\nu$ and $k$ are comparable when $\alpha \in [1/2, 1]$ and finally that we have bounds 
 $$
 |\Ai'(\nu^{2/3}\zeta)| + |\Bi'(\nu^{2/3} \zeta )| \leq C \nu^{1/6}
 $$
 uniformly for $\zeta$ in this range --- see \cite[10.4.62, 10.4.67]{AS1964}. It follows that 
 $$
 2 \arg H^{(1)}_{\alpha k}(k)  = 2 \Big( (\Ai - i \Bi)(\alpha^{2/3} k^{2/3} \zeta) \Big) + O(k^{-7/6}).
 $$
 Finally using the asymptotics \cite[10.4.60, 10.4.64]{AS1964}, we get 
 $$
 (\Ai - i \Bi)(\alpha^{2/3} k^{2/3} \zeta) = \frac{( \alpha k (-\zeta)^{3/2})^{-1/6}}{\pi^{1/2}} \bigg( (-i) e^{i \big( \frac{2}{3} \alpha k (-\zeta)^{3/2} + \pi/4 \big)} + O(k^{-1} (-\zeta)^{-3/2}) \bigg). 
 $$
 It follows, using the explicit expression for $\zeta(\alpha)$ in \eqref{zeta-alpha},  that
 $$
 2 \arg H^{(1)}_{\alpha k}(k)  = 2k \big( \sqrt{1-\alpha^{2} } - \alpha \cos^{-1}(\alpha) \big) - \pi/2 + O(k^{-1} (1 - \alpha)^{-3/2}). 
 $$
 Since we have taken $1 - \alpha \geq k^{-1/3}$, that gives us \eqref{ball-eval-approx}. 
 (Pleasingly, we get the same expression as in \eqref{smallalpha}, a useful check on the computations.)

We now turn to the proof of equidistribution. 
We first note that, as in the proof of Proposition~\ref{thm:realrealdiscrep} (with $S = \{ 1 \}$ and $\epsilon = k^{-1/3}$) the discrepancy of the exact eigenvalues $e^{ix_{k,l}}$ is equal to that of the approximate eigenvalues $e^{ikG_b(\alpha) - \pi/2}$ up to an error $O(k^{-1/3})$ which is acceptable. So it suffices to prove \eqref{eq:fastdecrease} for the approximations  $e^{ikG_b(\alpha) - \pi/2}$. 
%
%
We apply \eqref{eq:epsilondiscrep}, using
  \begin{equation}\label{eq:improvedkandr}
    \begin{split}
      \wt{\kappa} &\le \max_{0 \le \alpha \le 1} | G'  |\le \pi, \\
      \wt{\rho} &\ge \min_{0 \le \alpha \le 1} G''(\alpha) \ge c >
      0.
    \end{split}
\end{equation}
This
  means that in \eqref{eq:epsilondiscrep}, $\tilde{\rho}(\epsilon,
  1/k)   \ge c > 0$ and $\tilde\kappa(\epsilon, 1/k) \le \pi$ for
  all $\epsilon$, and thus for any $0 < \gamma < 1$,
  \begin{equation}\label{eq:realdiscrepdisck}
\begin{split}
  D(\mathcal{E}_{1/k}(\epsilon)) &\le C\lp  k^{-\gamma} + k^{-1/2 +
    \gamma/2} \rp.
\end{split}
\end{equation} 
Choosing $\gamma = 1/3$ completes the proof.

\bibliographystyle{abbrv}

\bibliography{sccentral}

\begin{thebibliography}{10}

\bibitem{AS1964}
M.~Abramowitz and I.~A. Stegun.
\newblock {\em Handbook of mathematical functions with formulas, graphs, and
  mathematical tables}, volume~55 of {\em National Bureau of Standards Applied
  Mathematics Series}.
\newblock For sale by the Superintendent of Documents, U.S. Government Printing
  Office, Washington, D.C., 1964.

\bibitem{A2005}
I.~Alexandrova.
\newblock Structure of the semi-classical amplitude for general scattering
  relations.
\newblock {\em Comm. Partial Differential Equations}, 30(10-12):1505--1535,
  2005.

\bibitem{BY1980}
M.~S. Birman and D.~R. Yafaev.
\newblock Asymptotics of the spectrum of the s-matrix in potential scattering.
\newblock {\em Soviet Phys. Dokl.}, 25(12):989--990, 1980.

\bibitem{BY1982}
M.~S. Birman and D.~R. Yafaev.
\newblock Asymptotic behavior of limit phases for scattering by potentials
  without spherical symmetry.
\newblock {\em Theoret. Math. Phys.}, 51(1):344--350, 1982.

\bibitem{BY1984}
M.~S. Birman and D.~R. Yafaev.
\newblock Asymptotic behaviour of the spectrum of the scattering matrix.
\newblock {\em J. Sov. Math.}, 25:793--814, 1984.

\bibitem{BY1993}
M.~S. Birman and D.~R. Yafaev.
\newblock Spectral properties of the scattering matrix.
\newblock {\em St Petersburg Math. Journal}, 4(6):1055--1079, 1993.

\bibitem{BP2011}
D.~Bulger and A.~Pushnitski.
\newblock The spectral density of the scattering matrix for high energies.
\newblock {\em arXiv:1110.3710}, 2011.

\bibitem{Sm1992}
E.~Doron and U.~Smilansky.
\newblock Semiclassical quantization of chaotic billiards: a scattering theory
  approach.
\newblock {\em Nonlinearity}, 5(5):1055--1084, 1992.

\bibitem{G1976}
V.~Guillemin.
\newblock Sojourn times and asymptotic properties of the scattering matrix.
\newblock In {\em Proceedings of the {O}ji {S}eminar on {A}lgebraic {A}nalysis
  and the {RIMS} {S}ymposium on {A}lgebraic {A}nalysis ({K}yoto {U}niv.,
  {K}yoto, 1976)}, volume~12, pages 69--88, 1976/77 supplement.

\bibitem{HW2008}
A.~Hassell and J.~Wunsch.
\newblock The semiclassical resolvent and the propagator for non-trapping
  scattering metrics.
\newblock {\em Adv. Math.}, 217(2):586--682, 2008.

\bibitem{H1971}
L.~H{\"o}rmander.
\newblock Fourier integral operators. {I}.
\newblock {\em Acta Math.}, 127(1-2):79--183, 1971.

\bibitem{Hvol1}
L.~H{\"o}rmander.
\newblock {\em The analysis of linear partial differential operators. {I}}.
\newblock Classics in Mathematics. Springer-Verlag, Berlin, 1983.

\bibitem{Hvol3}
L.~H{\"o}rmander.
\newblock {\em The analysis of linear partial differential operators. {III}}.
\newblock Classics in Mathematics. Springer-Verlag, Berlin, 2007.
\newblock Pseudo-differential Operators, Reprint of the 1994 edition.

\bibitem{KN1974}
L.~Kuipers and H.~Niederreiter.
\newblock {\em Uniform distribution of sequences}.
\newblock Wiley-Interscience [John Wiley \& Sons], New York, 1974.
\newblock Pure and Applied Mathematics.

\bibitem{LL1965}
L.~D. Landau and E.~M. Lifshitz.
\newblock {\em Quantum mechanics: nonrelativistic theory}.
\newblock Pergamon Press, Oxford, 1965.
\newblock Second revised edition.

\bibitem{Maj1976}
A.~Majda.
\newblock High frequency asymptotics for the scattering matrix and the inverse
  problem of acoustical scattering.
\newblock {\em Comm. Pure Appl. Math.}, 29(3):261--291, 1976.

\bibitem{GST}
R.~Melrose.
\newblock {\em Geometric Scattering Theory}.
\newblock Cambridge University Press, Cambridge, 1995.

\bibitem{MZ1996}
R.~Melrose and M.~Zworski.
\newblock Scattering metrics and geodesic flow at infinity.
\newblock {\em Invent. Math.}, 124(1-3):389--436, 1996.

\bibitem{Newton1966}
R.~Newton.
\newblock {\em Scattering Theory of Waves and Particles}.
\newblock McGraw-Hill, New York, 1965.

\bibitem{PZ2001}
V.~Petkov and M.~Zworski.
\newblock Semi-classical estimates on the scattering determinant.
\newblock {\em Ann. Henri Poincar\'e}, 2(4):675--711, 2001.

\bibitem{RSIII}
M.~Reed and B.~Simon.
\newblock {\em Methods of modern mathematical physics. {III}}.
\newblock Academic Press [Harcourt Brace Jovanovich Publishers], New York,
  1979.

\bibitem{RT1989}
D.~Robert and H.~Tamura.
\newblock Asymptotic behavior of scattering amplitudes in semi-classical and
  low energy limits.
\newblock {\em Ann. Inst. Fourier (Grenoble)}, 39(1):155--192, 1989.

\bibitem{SZ1991}
J.~Sj{\"o}strand and M.~Zworski.
\newblock Complex scaling and the distribution of scattering poles.
\newblock {\em J. Amer. Math. Soc.}, 4(4):729--769, 1991.

\bibitem{SY1985}
A.~V. Sobolev and D.~R. Yafaev.
\newblock Phase analysis in the problem of scattering by a radial potential.
\newblock {\em Zap. Nauchn. Sem. Leningrad. Otdel. Mat. Inst. Steklov. (LOMI)},
  147:155--178, 206, 1985.
\newblock Boundary value problems of mathematical physics and related problems
  in the theory of functions, No. 17.

\bibitem{Yafaev}
D.~Yafaev.
\newblock {\em Mathematical Scattering Theory: analytic theory}.
\newblock American Mathematical Society, Providence, RI, 2010.

\bibitem{Yafaev1990}
D.~R. Yafaev.
\newblock On the asymptotics of scattering phases for the {S}chr\"odinger
  equation.
\newblock {\em Ann. Inst. H. Poincar\'e Phys. Th\'eor.}, 53(3):283--299, 1990.

\bibitem{Z1997}
S.~Zelditch.
\newblock Index and dynamics of quantized contact transformations.
\newblock {\em Ann. Inst. Fourier (Grenoble)}, 47(1):305--363, 1997.

\bibitem{ZZ1999}
S.~Zelditch and M.~Zworski.
\newblock Spacing between phase shifts in a simple scattering problem.
\newblock {\em Comm. Math. Phys.}, 204(3):709--729, 1999.

\end{thebibliography}

\end{document}